\documentclass[10pt,a4paper]{amsart}

\usepackage{geometry} 
 \newgeometry{tmargin=2.5cm, bmargin=3cm, lmargin=2.5cm, rmargin=2.5cm}
\usepackage[T1]{fontenc}
\usepackage[notref,notcite]{}
\usepackage{graphicx}
\usepackage{enumerate}
\usepackage{amsmath,amsfonts,amssymb}
\usepackage{mathtools}
\usepackage{ulem}
\usepackage{color}
\usepackage{cite}
\definecolor{citation}{rgb}{0.2,0.58,0.2} 
\definecolor{formula}{rgb}{0.1,0.2,0.6}
\definecolor{url}{rgb}{0.3,0,0.5}
\usepackage{pgf,tikz}
\usepackage{mathrsfs}
\usepackage{fancyhdr}
\usepackage{dsfont}
\usepackage{esint}
\usepackage{hyperref}

\usepackage[centertags]{amsmath}
\usepackage{algorithm}
\usepackage{algpseudocode}
\usepackage{xfrac}
\usepackage{subcaption}
\usepackage{cleveref}
\usepackage{apptools}
\theoremstyle{plain}
\newtheorem{coro}{\bf Corollary}[section]
\newtheorem{theo}{Theorem}
\newtheorem{lem}[coro]{Lemma}
\newtheorem{prop}[coro]{Proposition}

\newtheorem{rem}[coro]{Remark}

\theoremstyle{definition}

\makeatletter
\def\namedlabel#1#2{\begingroup
    #2%
    \def\@currentlabel{#2}%
    \phantomsection\label{#1}\endgroup
}
\makeatother

\makeatletter
\newcommand{\settheoremtag}[1]{
  \let\oldthetheorem\thetheorem
  \renewcommand{\thetheorem}{#1}
  \g@addto@macro\endtheorem{
    \addtocounter{theorem}{-1}
    \global\let\thetheorem\oldthetheorem}
  }
\makeatother

\definecolor{darkgreen}{rgb}{0.00, 0.50, 0.00} 
\definecolor{airforceblue}{rgb}{0.36,0.54, 0.66}

\DeclareMathOperator*{\argmin}{arg\,min}

\newcommand{\nocontentsline}[3]{}
\newcommand{\tocless}[2]{\bgroup\let\addcontentsline=\nocontentsline#1{#2}\egroup}

\def\mean#1{\mathchoice%
          {\mathop{\kern 0.2em\vrule width 0.6em height 0.69678ex depth -0.58065ex
                  \kern -0.8em \intop}\nolimits_{\kern -0.4em#1}}%
          {\mathop{\kern 0.1em\vrule width 0.5em height 0.69678ex depth -0.60387ex
                  \kern -0.6em \intop}\nolimits_{#1}}%
          {\mathop{\kern 0.1em\vrule width 0.5em height 0.69678ex
              depth -0.60387ex
                  \kern -0.6em \intop}\nolimits_{#1}}%
          {\mathop{\kern 0.1em\vrule width 0.5em height 0.69678ex depth -0.60387ex
                  \kern -0.6em \intop}\nolimits_{#1}}}

\newcommand{\prox}{{\rm prox}}          
\newcommand{\dv}{{\rm div}}
\newcommand{\cP}{{\mathscr{P}}}
\newcommand{\cN}{{\mathscr{N}}}
\newcommand{\cF}{{\mathcal{F}}}

\newcommand{\cG}{{\mathcal{G}}}

\newcommand{\cFV}{{\cF_V}}
\newcommand{\cFW}{{\cF_W}}

\newcommand{\cFE}{{\cF_{\mathcal{E}}}}
\newcommand{\wt}{\widetilde}
\newcommand{\ve}{\varepsilon}
\newcommand{\vp}{\varphi}
\newcommand{\vt}{\vartheta}
\newcommand{\vr}{\varrho}

\newcommand{\rhok}{\vr_k}

\newcommand{\vrk}{\vr_k}
\newcommand{\vrki}{\vr_{k+1}}

\newcommand{\Xk}{X_k}
\newcommand{\Xki}{X_{k+1}}
\newcommand{\Xkip}{X_{k+\frac{1}{2}}}

\newcommand{\vrkN}{\vr_{k}^{N}}
\newcommand{\vrkp}{\vr_{k+\frac{1}{2}}}

\newcommand{\proxv}{\mathrm{prox}_V^{\tau}}
\newcommand{\Law}{{\rm Law}}

\newcommand{\Xpkph}{\widehat X_{k+\frac{1}{2}}}
\newcommand{\Xpkh}{\widehat X_{k}}
\newcommand{\Xpkoh}{{\widehat{X}_{k+1}}}
\newcommand{\muinfty}{{\overline{\mu}_\infty}}

\newcommand{\Xti}{X_t^i}
\newcommand{\Xtj}{X_t^j}
\newcommand{\Yti}{Y_t^i}

\newcommand{\Xri}{X_r^i}
\newcommand{\Xrj}{X_r^j}
\newcommand{\Yri}{Y_r^i}
\newcommand{\Yrj}{Y_r^j}
\newcommand{\Yrk}{Y_r^k}
\newcommand{\Xsi}{X_s^i}

\newcommand{\Ysi}{Y_s^i}

\newcommand{\Bti}{B_t^i}

\newcommand{\Yt}{Y_t}
\newcommand{\tYt}{\wt{Y}_t}
\newcommand{\Yr}{Y_r}
\newcommand{\Bt}{B_t}
\newcommand{\mut}{\mu_t}

\newcommand{\vtij}{\vartheta_{i,j}}
\newcommand{\vtik}{\vartheta_{i,k}}

\def\R{{\mathbb{R}}}
\def\Ex{{\mathbb{E}}}
\def\N{{\mathbb{N}}}
\def\Rd{{\R^d}}

\def\r{{\R}}

\def\dx{{\,\mathrm{d}x}}
\def\dy{{\,\mathrm{d}y}}
\def\ds{{\,\mathrm{d}s}}
\def\dmu{{\,\mathrm{d}\mu}}
\def\d{{\,\mathrm{d}}}

\def\dt{{\,\mathrm{d}t}}
\def\dr{{\,\mathrm{d}r}}
\def\dt{{\,\mathrm{d}t}}

\def\rn{{\R^{n}}}

\def\id{{\mathsf{Id}}}

\title[Convergence rates for granular medium equations]{Convergence rates of particle approximation\\ of forward-backward splitting algorithm\\ for granular medium equations}

\author{Matej Benko}\address{Matej Benko\\ 
Institute of Mathematics, Faculty of Mechanical Engineering, Brno University of Technology\\ Technická 2896/2, 616 69 Brno, Czech Republic \\
\texttt{e-mail: matej.benko@vutbr.cz}}

\author{Iwona Chlebicka}\address{Iwona Chlebicka \\
Institute of Applied Mathematics and Mechanics, University of Warsaw \\ ul. Banacha 2, 02-097 Warsaw, Poland\\  \texttt{e-mail: i.chlebicka@mimuw.edu.pl}} 

\author{J\o{}rgen Endal}\address{J\o{}rgen Endal\\
Department of Mathematical Sciences, Norwegian University of Science and Technology (NTNU)\\
N-7491 Trondheim, Norway\\ \texttt{e-mail: jorgen.endal@ntnu.no}}

\author{B\l{}a\.z{}ej Miasojedow}\address{B\l{}a\.z{}ej Miasojedow \\
Institute of Applied Mathematics and Mechanics, University of Warsaw \\ ul. Banacha 2, 02-097 Warsaw, Poland\\  \texttt{e-mail: b.miasojedow@mimuw.edu.pl}}

\begin{document}

\begin{abstract}
We study the spatially homogeneous granular medium equation 
\[\partial_t\mu=\dv(\mu\nabla V)+\dv(\mu(\nabla W \ast \mu))+\Delta\mu\,,\] 
within a large and natural class of the confinement potentials $V$ and interaction potentials $W$. The considered problem do not need to assume that $\nabla V$ or $\nabla W$ are globally Lipschitz.

With the aim of providing particle approximation of solutions, we design efficient forward-backward splitting algorithms. Sharp convergence rates in terms of the Wasserstein distance are provided.

\end{abstract}

\maketitle
\setcounter{tocdepth}{1}
\tableofcontents

\section{Introduction}

\normalcolor

We show sharp convergence rates in terms of the Wasserstein distance of a new forward-backward splitting algorithm for the following {\it spatially homogeneous} granular medium equation
\begin{equation}
\label{eq:model} 
\begin{cases}\partial_t\mu=\dv(\mu\nabla V)+\dv(\mu(\nabla W \ast \mu))+\Delta\mu &\ \text{in }\ \R^d\times(0,T)\,,\\
\mu(\cdot,0)=\mu_0&\ \text{in }\ \R^d\,
\end{cases}
\end{equation}
with arbitrary datum $\mu_0$  with finite second moment. In the one-dimensional case, this equation  was introduced in a mathematical context to model granular media in \cite{BCP,BCCP}. Therein $\mu(\cdot,t)$ describes the distribution of the velocity of a representative particle (among infinitely many), $V$ represents the friction (depending on the application it is also called the confinement potential), $W$ models
the inelastic collisions between particles with different velocities (also called the interaction potential), while $\Delta \mu$ accounts for a so-called heat bath (or diffusion).  We refer to the survey \cite{Villani2006} (see also the very detailed \cite{Villani-review}) where the above equation with $V(x)=\frac{1}{2}|x|^2$ and $W(x)=\frac{1}{3}|x|^3$ appears in Section 3. There the reader will find a thorough exposition on the physical derivation of the equation as a limit of a  system of many infinitesimal particles colliding inelastically. There is also a recent survey \cite{Gomez-Castro2024} which also considers \eqref{eq:model}, but with possibly different applications in mind. In this broader context, it is called an aggregation-confinement-diffusion equation. Other labels includes nonlinear Fokker--Planck equations, see e.g. \cite{BCCP, surv-I, surv-II, DiMarino-Santambrogio}, and nonlinear McKean--Vlasov PDEs \cite{MALRIEU_log_sobolev}.

The granular medium equation models interactions among particles, incorporating aggregation, diffusion, and self-propulsion mechanisms \cite{CMCV,Villani-review,BCP,BGG}. This direction of study bridges microscopic interactions and macroscopic behavior, influencing physics, engineering, and material science.  A great number of recent contributions were devoted the dynamics of equations like~\eqref{eq:model}, revealing clustering and pattern formation. Our approach employs splitting algorithms that take an advantage of decomposing a relevant function into the sum of two convex functions, alternating between minimizing one function while holding the other fixed.  Noteworthy, the particle approximation of our forward-backward splitting algorithm is proven to keep sharp control on the convergence rates in the terms of the Wasserstein distance.  

To provide the convergence rates, we will use the viewpoint of statistical mechanics. Let us therefore consider a system of $N$ {\it linear} SDEs (stochastic differential equations) on the form 
\begin{equation}\label{eq:particle_system:model}
\begin{cases}
\d \Xti = -\Big(\nabla V(\Xti) + \frac{1}{N}\sum_{\substack{j=1 \\ i \neq j}}^{N} \nabla W (\Xti - \Xtj)\Big) \dt + \sqrt{2 } \, \d \Bti &\ \text{$i=1, \ldots , N$, $t>0$}\,, \\ 
X_0^i\sim\mu_0 &\ \text{$i=1, \ldots , N$}\,, 
\end{cases}
\end{equation}
where $B_t^1,\ldots,B_t^N$ are independent $d$-dimensional standard Brownian motions (Wiener processes). We will consider the case where we relate a particle with each equation, and the particles interact due to the presence of $W$. This gives a particle system $\boldsymbol{X}_t =(X_t^1,\ldots,X_t^N)$ which has law given by $\boldsymbol{\mu}_t \in \cP_2^{\mathrm{ex}} (\R^{dN})$ (exchangeable probability measures with finite second moment) such that 
\begin{equation}\label{Xt,X0-distribution} 
\boldsymbol{X}_t \sim \boldsymbol{\mu}_t \qquad \mbox{and} \qquad \boldsymbol{X}_0 \sim \bigotimes_{i = 1}^N \mu_0  =: \boldsymbol{\mu}_0\, .
\end{equation}
As the number of particles in the system like~\eqref{eq:particle_system:model} tends to infinity, the behavior of any finite subset of particles is expect to become increasingly independent of the rest of the system. Convergence of the empirical measure of the particle system (the distribution function of particles) to a deterministic limit as the number of particles tends to infinity, called {\it
propagation of chaos}, is studied  since~\cite{Kac,Sznitman} and surveyed in~\cite{surv-I,surv-II}. By this idea we approximate  the solution
of the nonlinear PDE~\eqref{eq:model} by the empirical measure of the particle system of linear SDEs \eqref{eq:particle_system:model}, cf.~\cite{CGM,MALRIEU_log_sobolev}. If we restrict our attention to one of these linear SDEs, say equation $j$, we can prove that $X_t^j$ converges strongly as $N\to+\infty$ to the solution of the following {\it nonlinear} SDE of McKean--Vlasov type
\begin{equation} 
\label{eq:mckean_vlasov:model}
\begin{cases}
\mathrm{d} \Yt = -\big(\nabla V(\Yt) + (\nabla W * \mu_t)(\Yt) \big)\dt + \sqrt{2 } \, \d \Bt\,,\qquad \text{$t>0$}\,, \\ 
\mu_t =\Law(\Yt)\,, \\
Y_0\sim \mu_0\,.
\end{cases}
\end{equation} Strictly speaking, using synchronous coupling~\cite[Section 4.1.2]{surv-I} we show that  $X_t^i$ converges strongly as $N\to+\infty$ to $Y_t^i$, for all $i$, where $Y_t^i\sim \mu_t$ is an independent copy solving the above equation with the Brownian motion $\Bt^i$. Moreover, since the $\{Y_t^i\}_i$ are independent copies, we have that $\mu_t^{i,N}\to\mu_t$ as $N\to+\infty$ for every $i=1,\ldots,N$. Hence, we will use the short-hand notation $\cP_2^{\mathrm{ex}} (\R^{d})\ni\mu_t^N:=\mu_t^{i,N}$ for every $i$. We refer to Theorem \ref{theo:chaos} for precise statements. Finally, the density $\mu_t$ is the solution of the nonlinear PDE 
\eqref{eq:model} under study.

In order to construct a numerical scheme to solve \eqref{eq:model} we shall follow the gradient flow approach to PDEs, cf.~\cite{Ambrosio-Gigli-Savare,Figalli-Glaudo,Santambrogio-overview}.  In this approach, the evolution of particle density within the granular medium is modeled as a gradient flow, where the system tends to decrease its energy functional over time.   We define the following functional for $\mu\in\cP(\R^d)$ (probabilistic measures):\begin{equation}
    \label{cF}
    \cF[\mu]:=\cFV[\mu]+  \cFE[\mu]+\cFW[\mu]\,,
\end{equation}
where (denoting $\dmu(x)=\mu(x)\dx$ when a density exists)\begin{flalign}
 \label{cFV}   \cFV[\mu]&:=\int_{\R^d}V(x)\dmu(x)\,,\\
 \label{cFE}   \cFE[\mu]&:=\begin{cases}
        \int_{\R^d}\mu(x)\log\mu(x)\dx\,, &\text{if }\mu\in\cP_{\rm ac}\,,\\
        +\infty\, ,&\text{otherwise}\,,
    \end{cases}\\
 \nonumber  \cFW[\mu]&:=\frac{1}{2}\int_{\R^d}\int_{\R^d} W(x-y)\dmu(x)\otimes\mu(y)\,.
\end{flalign}
We will call $\cFV$ the potential energy, $\cFE$ -- the entropy, while  $\cFW$ -- the nonlocal interaction. The gradient flow perspective enables the development of efficient numerical algorithms based on forward-backward gradient methods for solving granular medium equations.

Let us briefly explain our main idea. In order to reduce the problem, we make use of the approximation of~\eqref{eq:model} by~\eqref{eq:particle_system:model}, obtained via the propagation of chaos approach that allows to let $W\equiv0$. In turn, the corresponding energy, in the gradient flow perspective, consists of the two parts $\cFV$ and $\cFE$, while $\cFW=0$. We alternate steps: one along the gradient flow of $\cFV$ and the next along the gradient flow of $\cFE$. Both gradient flows given by $\cFV$ and $\cFE$ enjoy explicit analytical formulation.  The gradient flow of $\cFV$ is a pushforward of the measure by the proximal operator and the gradient flow of $\cFE$ is given by the Gaussian distribution. This technique is called a splitting algorithm and dates back to~\cite{LM79}. It is known in the context of SDE and PDE, see e.g. \cite{HMS,MS,DMM}. We point out that our main idea is closely related to the proximal gradient method~\cite{Prox-alg}, which was introduced in order to solve optimization problems expressed as variational inequalities. Our approach can be also viewed as a discretization of the Jordan--Kinderlehrer--Otto (JKO)~\cite{JKO} scheme that offers several advantages for solving variational problems associated with the granular medium equation. Firstly, it provides a computationally efficient and numerically stable method for approximating solutions to these problems. Secondly, the JKO scheme naturally incorporates the underlying convex structure of the problem, allowing for straightforward implementation. This approach thus allows us to simulate and analyze the behaviour of granular materials over time and space accurately. 
We provide the first computational method for PDE~\eqref{eq:model} under no global Lipschitz condition on the potentials and also taking into account possible non-exactness of the proximal step.

Numerical methods for solving variants of  equation~\eqref{eq:model} are widely studied in the literature. In the linear case ($W\equiv0$), rates of convergence are indeed obtained for the Euler--Maruyama scheme \cite{KlPl}, the finite volume scheme \cite{Mer0708,N-FSc23}, and the finite difference scheme \cite{KaRi01, ChKa06, KaKoRi12}, but under the assumption of Lipschitz continuity of $\nabla V$. Rates are also obtained for the forward-backward scheme under only a one-sided Lipschitz condition on $\nabla V$ \cite{HMS}. In the case when $W\not\equiv0$, 
there exist several approaches to problems of this type including the blob method~\cite{CCP,KT,CEHT}, the finite volume method \cite{CCY,BCH,CFS,BCJ}, and the primal-dual method~\cite{Carrillo-primal}. All mentioned methods allow for more general diffusion operators. The full discretization of the JKO scheme is the main tool in \cite{Carrillo-primal}. Therein, by leveraging the duality properties, one can design efficient algorithms that update both primal and dual variables iteratively, leading to convergence to the  solution. While our particle approximation approach focuses on simulating the behavior of individual particles within the granular medium, primal-dual methods provide a framework for designing algorithms to solve the associated optimization problem. There is also a relation with scalar conservation laws with nonlocal flux (usually with $d=1$ and without diffusion). We do not intend to make a full picture of that field, but we mention that the finite volume method for entropy solutions is e.g. studied in \cite{AmCoTe15, AgHoVa24} under the assumption that $\nabla W$ is at least Lipschitz continuous. The latter also provides the only rates of convergence we are aware of in this setting, assuming that $\partial_x W\in C^2\cap W^{2,\infty}\cap BV$. Hence, their setup is very different from ours. Recent related results in the SDE setting \cite{ChDR24,ChDRStWi25} are commented on right after Theorem \ref{thm:alg_bound-nonlocal}. 

We also refer to  \cite{Salim-Korba-Luise} for some ideas, that inspired us, which were used therein in the optimization algorithm for a functional over the Wasserstein spaces. In the case of the functional related to~\eqref{eq:model}, their method establishes the approximation of the stationary solution. Such a goal attracts significant attention in the context of Bayesian inference or Machine Learning. The discretization of the linear model ($W\equiv0$) is called Langevin Monte Carlo algorithm and among others was studied in \cite{DM, DMM, Dal-2017,DMP-2018,bernton}, including rates of convergence. At this point, let us also mention the recent contribution \cite{BeChEnMi25}, which provides the theoretical foundation for such algorithms also in the case of non-globally Lipschitz $\nabla V$. Recently in  Machine Learning problems, approximation of the stationary distribution of the full problem~\eqref{eq:model} using the advantages of propagation of chaos, was investigated in \cite{kook2024sampling}.\newline

Our contributions are uniform-in-time rates of convergence of a numerical scheme for~\eqref{eq:model} with fully explicit, computable, and polynomially dependent on $d$ constants in the following two cases:
\begin{itemize}
\item[{\it (i)}] Algorithm \ref{alg:dgf} where $W\equiv 0$ and $\nabla V$ is non-globally Lipschitz.
\item[{\it (ii)}] Algorithm \ref{alg:particle} where $W\not \equiv 0$ and $\nabla V, \nabla W$ are non-globally Lipschitz.  
\end{itemize}
The key mathematical challenges in obtaining the above results were the lack of satisfactory moment bounds for solutions to~\eqref{eq:mckean_vlasov:model} and for the output of Algorithm~\ref{alg:dgf}. We provide improved moment bounds for the
continuous time process (Theorem~\ref{thm:cont-moments}) and novel ones for the related discrete time process (Theorem~\ref{theo:moments}). Our main accomplishments are summarized in Fig.\ref{fig:all-results}. 
Before formulating them we note that we provide our main results assuming that the relevant solutions exist, see Section~\ref{ssec:ex-n-uni} for a discussion on available results in this direction.
Our main assumptions read:
\begin{description}
\item[\namedlabel{V}{(V)}]  $V\in C^1(\R^d,\R)$ has a minimum at $x^*=0$, is $\lambda_V$-convex for $\lambda_V>0$, and $\nabla V$ is of $q_V$-power growth Lipschitz for $q_V\geq 1$.

\item[\namedlabel{W}{(W)}] $W\in C^1(\R^d,\R)$ is convex, radially symmetric, and $\nabla W$ is of $q_W$-power growth  Lipschitz for $q_W\geq 1$.

\item[\namedlabel{mu0}{($\boldsymbol{\mu_0}$)}] The initial datum $\mu_0\in\cP_2(\r^d)$ also has an $m_0$-moment bound, i.e., $\mu_0(|\cdot|^{m_0})<+\infty$ for some $m_0$ specified later to be large enough and dependent on $V$ and $W$. 
\end{description}

\begin{figure}
\pdfimageresolution=100
\includegraphics[width=\textwidth]{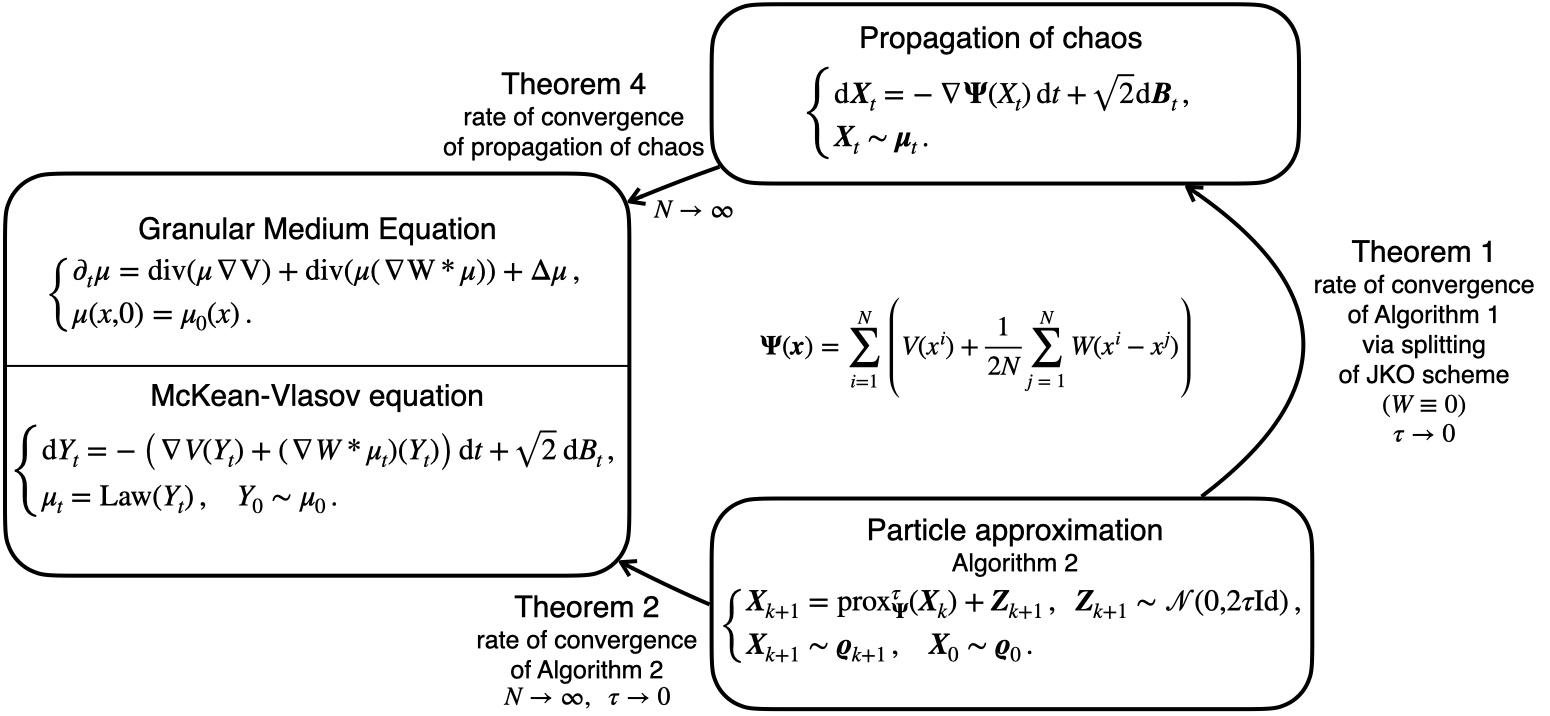}
    \caption{Summary of our approach positioning the content of Theorem~\ref{thm:alg_bound-intro} provided for Algorithm~\ref{alg:dgf}, together with Theorems~\ref{thm:alg_bound-nonlocal} and~\ref{theo:chaos} for Algorithm~\ref{alg:particle}. }   
    \label{fig:all-results}
\end{figure}

The following comments on our assumptions are in order.
\begin{enumerate}[{\it (i)}]

\item In the Appendix, one can find precise definitions of $\lambda$-convexity (see~\eqref{lambda-conv}) and $q$-power growth Lipschitz property for $q\geq 1$ (see~\eqref{f-q-growth-Lip}). In the latter case, we remark that the assumption on the $q$-power growth Lipschitzianity for $q>1$ does not require the functions to be globally Lipschitz (i.e. locally Lipschitz is enough).

\item Our assumptions on $V,W$ include potentials on the form $|x|^a$ with $a>1$, i.e., the attractive ones.

\item Without loss of generality, we can indeed assume that the minimum of $V$ is at $x^*=0$. Indeed, if this was not the case, we could consider $\widetilde{V}(x)=V(x+x^*)$. Then $\widetilde{\vr}(x,t)=\vr(x+x^*,t)$ satisfies a PDE of the same form as \eqref{eq:model}, but with $\widetilde{V}$, and $\wt\Yt=Y_t-x^*$ satisfies a McKean--Vlasov SDE of the same form as \eqref{eq:mckean_vlasov:model}, but with $\wt V$.
\item The assumptions on the differentiability of $V$ and $W$ can be omitted, and the gradient can be replaced by a subgradient. However, the theory about existence and uniqueness of the solution to \eqref{eq:model} in general is not known. We show results assuming $V,W\in C^1(\R^d,\R)$ to simplify the presentation. 
\end{enumerate}

\noindent{\bf Algorithms.} We consider two algorithms to numerically compute the gradient flow $\mu_t$ along \eqref{eq:model} in the cases $W\equiv0$ and $W\not\equiv0$, respectively. By $Y_t \sim \mu_t$ we mean the true solution of the SDE \eqref{eq:mckean_vlasov:model} whose density is the gradient flow of \eqref{eq:model}.

In the case $W\equiv0$, we will, for $k=1,\ldots,n$, denote by $X_k \sim \vr_k$ the result of our algorithm. Additionally, $X_0 \sim \vr_0=\mu_0$ and the time step $\tau$ will be chosen such that $T=n\tau$.  
Thus, we propose the algorithm based on the operator splitting technique consisting of two steps: 
\begin{enumerate}[(1)]
    \item Application of the proximal operator ($\proxv$ is defined in~\eqref{def:proximal}).
    \item Adding Gaussian noise $g_{\tau}$.
\end{enumerate}
See the scheme proposed in Algorithm~\ref{alg:dgf}, whose rates of convergence are given by Theorem~\ref{thm:alg_bound-intro}. Moreover, Theorem~\ref{theo:moments} establishes a bound on the moments of the Markov Chain produced by the algorithm. Note also that the pushforward $\#$ is defined in \eqref{eq:pushforward}.

\begin{algorithm}
\caption{Proximal Split Discretized Gradient Flow}\label{alg:dgf}
\begin{algorithmic}
\State Sample initial distribution $X_0 \sim \vr_0=\mu_0$
\For{$k=0, \ldots , n-1$}
\State  $\begin{array}{ll}
(1) \quad X_{k+\frac{1}{2}} := \prox^\tau_V (X_k)&\qquad\triangleright\quad \vr_{k+\frac{1}{2}} := (\prox^\tau_{V})_{\#} \vr_k \\
(2) \quad X_{k+1} := X_{k+\frac{1}{2}} + Z_{k+1}; \quad Z_{k+1} \sim g_{\tau} & \qquad\triangleright\quad \vr_{k+1} := \vr_{k+\frac{1}{2}} * g_{ \tau}
\end{array}$
\EndFor
\end{algorithmic}
\end{algorithm}

In the second case, namely for $W\not\equiv0$, we use the {\it propagation of chaos} to replace \eqref{eq:mckean_vlasov:model} with the system of $N$ linear SDEs \eqref{eq:particle_system:model} in which any marginal of the solution (which is assumed to be exchangeable) converges to $\mu_t$ as $N \to +\infty$, where $\mu_t$ solves \eqref{eq:model}. Let us define $\boldsymbol{\Psi}:\R^{dN}\to \R^N$  via the following formula
\begin{equation}\label{def-PsiN}
   \boldsymbol{\Psi}(\boldsymbol{x}):= \sum_{i=1}^N\left(V(x^i) +  \frac{1}{2N}\sum_{\substack{j=1}}^{N}  W (x^i-x^j)\right) \qquad\text{for }\ \boldsymbol{x}\in\R^{dN}\,. 
\end{equation}
We note that if $V$ attains minimum at $x^*=0$, then the same can be assumed for $\boldsymbol{\Psi}$. By Lemmas~\ref{lem:Psi1} and~\ref{lem:Psi2}, assumptions \ref{V} and \ref{W} implies that $\boldsymbol{\Psi}$ also satisfies \ref{V}. Using this notation the system of linear SDEs \eqref{eq:particle_system:model} can be written as follows
\begin{equation}\label{eq:particle_system}
\d \Xti = 
-\partial_{x^i}\boldsymbol{\Psi}(\boldsymbol{X}_t)\d t+ \sqrt{2 } \, \d \Bti\,\,,\qquad i=1, \ldots , N \, .
\end{equation}
At this point, we see that $\boldsymbol{\Psi}$ plays the role of $V$ in \eqref{eq:mckean_vlasov:model} (with $W\equiv0$), and thus we use the idea of Algorithm \ref{alg:dgf} to design an algorithm which approximates the linear system of SDEs, that will be called Algorithm~\ref{alg:particle}. To this end, with a slight abuse of notation, for $k=1,\ldots,n$, we denote by $\boldsymbol{X}_k \sim \boldsymbol{\vr}_k$ the result of our new algorithm. Additionally, $\boldsymbol{X}_0 \sim \boldsymbol{\vr}_0=\boldsymbol{\mu}_0$ (see \eqref{Xt,X0-distribution}) and the time step $\tau$ will be chosen such that $T=n\tau$.  Convergence rates for Algorithm~\ref{alg:particle} are provided in Theorem~\ref{thm:alg_bound-nonlocal}.
\begin{algorithm}
\caption{Proximal Split Discretized Particle Approximation of the Gradient Flow}\label{alg:particle}
\begin{algorithmic}
\State Sample initial distribution $\boldsymbol{X}_0 \sim \boldsymbol{\vr}_0=\boldsymbol{\mu}_0$
\For{$k=0, \ldots , n-1$}
\State  $\begin{array}{ll}
(1) \quad \boldsymbol{X}_{k+\frac{1}{2}} := \prox^\tau_{\boldsymbol{\Psi}} (\boldsymbol{X}_k)&\qquad\triangleright\quad \boldsymbol{\vr}_{k+\frac{1}{2}} := (\prox^\tau_{\boldsymbol{\Psi}})_{\#} \boldsymbol{\vr}_k \\
(2) \quad \boldsymbol{X}_{k+1} := \boldsymbol{X}_{k+\frac{1}{2}} + \boldsymbol{Z}_{k+1}; \quad \boldsymbol{Z}_{k+1} \sim \boldsymbol{g}_{\tau} & \qquad\triangleright\quad \boldsymbol{\vr}_{k+1} := \boldsymbol{\vr}_{k+\frac{1}{2}} * \boldsymbol{g}_{\tau}
\end{array}$
\EndFor
\end{algorithmic}
\end{algorithm}

\bigskip

\noindent{\bf Main result for the local problem (when $W\equiv0$).} Let us present our first result on the error bounds for our discretization of solutions $\mu$ to \eqref{eq:model}. We stress that we provide a bound for $ W_2^2 (\vrk, \mu(t))$ which is uniform in time, polynomial with dimension $d$ (with an exponent governed by the growth of potential~$V$), and linear with respect to the length of discretization step $\tau$.

\begin{theo} \label{thm:alg_bound-intro}
    Let $\overline{\mu}_\infty$ be the unique minimizer of $\cF$ given by~\eqref{cF} (with $W\equiv0$) for $V$ that satisfies \ref{V} with $\lambda_V > 0$ and $q_V\geq 1$. Assume further that $\mu$ is a gradient flow solution to~\eqref{eq:model} with $W\equiv 0$ and initial datum $\mu_0\in\cP_2(\R^d)\cap \cP_{q_{V}-1}(\R^d)$. Let $\vr_k$ be an output of Algorithm~\ref{alg:dgf} for every $k=1,\ldots,n\,$.    
    Then there exist $c_1,c_2,\tau_0>0$ depending only on $V,\mu_0,\overline{\mu}_\infty$, such that for all $\tau\in( 0,\tau_0)$, all $k=2,\dots,n$, and $t\in((k-1)\tau, k\tau ]$ the following inequality holds true
\[
W_2^2 (\vrk, \mu(t)) \leq \min\left\{c_1(1+t)d^{\frac{q_V+1}{2}}\tau, c_2 \left(\mathrm{e}^{-\frac{\lambda_V}{2}t}+ d^{\frac{q_V+1}{2}}\tau\right)\right\}\, .
\]
\end{theo}

We have the following comments on the above result.
\begin{enumerate}[{\it (i)}]

\item In the case $k=1$, we have the estimate
\[
W_2^2 (\vr_{1}, \mu(t)) \leq c_1d^{\frac{q_V+1}{2}}(1+\tau)\tau\qquad\text{for }\ 0<t\leq \tau\,.
\]
\item The explicit expressions for the constants are
\[
c_1=6\big(\cF [\mu_0] - \cF [\muinfty]\big) +  6C \big(1+{\lambda_V}^{-\frac{q_V-1}{2}}\big)
\]
and
\[
c_2=4W_2^2(\mu_0,\overline{\mu}_\infty)+C\lambda_V^{-1}\big(1+\lambda_V^{-\frac{q_V-1}{2}}\big)\,,
\]
where $C=C\big(q_V,\mu_0(|\cdot|^{q_V-1})\big)$.
\item Typically Lipschitz continuity of $\nabla V$ is required to stabilize a discretization of a step corresponding to the $V$-part. Thanks to the stability of implicit discretization provided by proximal step, we can handle also non-Lipschitz potentials.
\item Under the regime of Theorem~\ref{thm:alg_bound-intro} there exists a unique minimizer $\muinfty$ of $\cF$, which is the stationary measure
of \eqref{eq:mckean_vlasov:model} for $W\equiv 0$ given by $\muinfty(x)\propto e^{-V(x)}$. Indeed, lower semicontinuity and $\lambda$-convexity of $\cF$ follows from~\cite{Ambrosio-Gigli-Savare}, see Examples~9.3.4 and~9.3.6 together with Remark~9.3.8 therein. Therefore, the result 
follows from Lemma~\ref{lem:minimizer}, see also e.g. \cite{BGG}. We need $\muinfty$ in Lemma \ref{lemma:dmm3}, and it also appears in Theorem 4.1 in \cite{bernton} which we use to prove the above theorem. To get a bound on $ W_2^2 (\vrk, \mu(t))$ which is uniform in time, we basically separate between to cases: $t$ close to zero, and $t$ close to infinity. For $t$ close to infinity, we use that the solution $\mu(t)$ is close to $\muinfty$ in the sense $W_2^2(\mu(t),\muinfty)\leq e^{-2\lambda_Vt}W_2^2(\mu_0,\muinfty)$ due to the $\lambda_V$-contraction for gradient flow solutions of  \eqref{eq:model} (cf. \cite[Theorem 11.2.1]{Ambrosio-Gigli-Savare}).
\item The proximal operator is typically infeasible in practice. However, one can approximate the proximal operator and obtain the same convergence rates, see Theorem~\ref{theo:linear-problem-stability}.
\item {Our bounds are sharp in the sense that we get the same dependence on $\tau$ as in  \cite[Theorem~4.0.10]{Ambrosio-Gigli-Savare} when the initial measure does not satisfy any additional assumptions.}
\end{enumerate}

\normalcolor

\medskip
\noindent{\bf Main result for the problem with nontrivial nonlocal interactions (when $W\not\equiv0$).} We have the following estimate on the convergence rates of particle approximation of forward-backward splitting algorithm for granular medium equation in the general case~\eqref{eq:model}.  At this point we also remind the reader that we denote $\vr_k^N=\vr_k^{i,N}$ for any $i$ where $\boldsymbol{X}_k\sim \boldsymbol{\vr}_k=(\vr_k^{1,N},\ldots,\vr_k^{N,N})$ is an output from Algorithm~\ref{alg:particle}.

\begin{theo} \label{thm:alg_bound-nonlocal}
    Let $\muinfty$ be the unique minimizer of $\cF$ given by~\eqref{cF} for $V,W$ that satisfy \ref{V} and \ref{W} with $\lambda_V > 0$, $q_V,q_W\geq 1$. Let $\wt q_1=\max\{q_V,q_W\}$ and $\wt q_2=\max\{q_V,2q_W-1\}$. Assume further that $\mu$ is a gradient flow solution to~\eqref{eq:model} with initial datum $\mu_0\in\cP_2(\R^d)\cap \cP_{{\wt q_2}}(\R^d)$ and $\boldsymbol{\vr}_k$ is an output of Algorithm~\ref{alg:particle}  for every $k=1,\ldots,n$ and $N\geq 1$.  Then there exist $c_1,c_2,c_3,\tau_0>0$ depending only on $V,W,\mu_0, \overline{\mu}_\infty$, such that for all $\tau\in( 0,\tau_0)$, all $k=2,\dots\,$, and $t\in(\tau(k-1), \tau k]$ the following inequality holds true
\begin{flalign*}
W_2^2 (\vrkN, \mu(t)) \leq \min &\left\{c_1(1+t)(dN)^{\frac{{\wt q_1+1}}{2}}\tau, c_2 \left(\mathrm{e}^{-t\frac{\lambda_V}{2}}+  (dN)^{\frac{{\wt q_1+1}}{2}}\tau \right)\right\}+c_3 {d^{2(q_W-1)}} N^{-1} \, .
\end{flalign*}
\end{theo}

Let us comment on the above result.
\begin{enumerate}[{\it (i)}]
\item By taking $\tau\leq \min\{N^{-1-\frac{{\wt q_1+1}}{2}},\tau_0\}$, we obtain that $W_2^2 (\vrkN, \mu(t)) \leq \mathcal{O}(N^{-1})$.
\item 
In order to get precision $\epsilon$ in $W_2$, it requires to perform  $\mathcal{O}(\epsilon^{-3-\wt q_1})$ steps  of Algorithm~\ref{alg:particle} with  $N=\mathcal{O}(\epsilon^{-2})$ particles. Hence, the total computational cost is of order $\mathcal{O}(\epsilon^{-5-\wt q_1})$.
\item In the case $k=1$, we have the estimate 
\[
W_2^2 (\vr_{1}^N, \mu(t)) \leq c_1(dN)^{\frac{\widetilde{q}_1+1}{2}}(1+\tau)\tau\qquad\text{for }\ 0<t\leq \tau\,.
\]
\item The constants $c_1,c_2$ are twice larger than the constants from Theorem \ref{thm:alg_bound-intro} (with $\tilde{q}_1$ replacing $q_V$), while $c_3$ is twice larger than the constant from Theorem \ref{theo:chaos}.
\item Let us again remark that there exists a unique minimizer $\muinfty$ of $\cF$ by Lemma~\ref{lem:minimizer}, even in the current case when $W\not\equiv0$, now given by the implicit formula $\muinfty(x)\propto e^{-V(x)-(W\ast\muinfty)(x)}$. Since we assume that $\mu(t)$ is a gradient flow solution, we still have $W_2^2(\mu(t),\muinfty)\leq e^{-2\lambda_Vt}W_2^2(\mu_0,\muinfty)$.
\item At this point it is interesting to note that our techniques do not rely on the presence of $\Delta \mu$. In fact, we can obtain exactly the same result for the equation $\partial_t\mu=\dv(\mu\nabla V)+\dv(\mu(\nabla W \ast \mu))$.
\end{enumerate}

Let us comment on some recent results within the pure SDE setting \cite{ChDR24,ChDRStWi25}. The setup of \cite{ChDR24} embraces McKean--Vlasov SDEs related to \eqref{eq:model} under assumptions as ours, but they also incorporate some terms we do not consider. On the other hand, \cite[Theorem 2.19]{ChDR24} provides rates of convergence for the split-step method involving bounds that are not explicit and not uniform in time. Moreover, both contributions prove propagation of chaos: Theorem~\ref{theo:chaos} yields bound of order $N^{-1/2}$, while \cite[Proposition~2.5]{ChDR24} provides the same order only when $d<4$ (for $d\geq4$ it gets worse in $N$). Under stronger (global Lipschitz) assumptions, better rates for a non-Markovian Euler-type scheme is provided in a related one-dimensional case in \cite{ChDRStWi25}.

Finally, we mention Theorem~\ref{thm:cont-moments} providing a result on controlling the moments of solutions to McKean--Vlasov SDE as a valuable tool for analyzing the behavior, stability, and convergence properties of stochastic systems, cf. \cite{MALRIEU_log_sobolev,CGM}. This result is of a separate interest not only for further mathematical analysis of the equation, but also it might be used in validating models in various applications.\medskip

\noindent {\bf Organization. } We start our paper with presenting notation, preliminary information on the well-posedness on the studied problems, and key properties of the setting in Section~\ref{sec:prelim}. Section~\ref{sec:nonlocal-problem} is devoted to the moment control and the propagation of chaos (Theorems~\ref{thm:cont-moments} and \ref{theo:chaos}).   In Section~\ref{sec:local-problem} we deal with a local problem, when $W\equiv 0$. First we establish Theorem~\ref{theo:moments} and our key technical tool of Lemma~\ref{lemma:dmm3}.  Then we present proofs of our main results, namely Theorems~\ref{thm:alg_bound-intro} and~\ref{thm:alg_bound-nonlocal}. Section~\ref{sec:5} focuses on approximation of the proximal operator within our approach (Algorithm~\ref{alg:dgf-pert}) and its convergence rates (Theorem~\ref{theo:linear-problem-stability}). We also explain how the perturbed Algorithm~\ref{alg:dgf-pert} can be realized via the standard gradient descent algorithm. Discussion on possible extensions of the presented methods is exposed in Section~\ref{sec:6}. We illustrate our results by numerical examples in Section~\ref{sec:7}. Some auxiliary facts are proven in Appendix.

\section{Preliminaries}\label{sec:prelim}
\subsection{Notation and basics}
 To abbreviate expressions in the paper, we introduce the notation 
\begin{equation*}
f(x) \leq \mathcal{O}(x) \qquad \mbox{iff} \qquad \exists C \in (0, + \infty); \ |f(x)| \leq C|x| \quad  \forall x \in \R,
\end{equation*}
and also 
\begin{equation*}
f(x) = o(x) \qquad \mbox{iff} \qquad f(x) \to 0 \quad \mbox{as} \quad x \to 0 \, .
\end{equation*}
We denote by $g_{a}$, $a>0$, the Gaussian noise (measure) such that $g_{a} \sim \cN(0, 2a \, \mathsf{Id})$. It is explicitly given as $g_a(x)=\frac{\textup{exp}(-|x|^2/(4a))}{(4\pi a)^{d/2}}$. We will repeatedly use the following fundamental inequalities:
$$
|x+y|^{a}\leq 
\begin{cases}
|x|^a+|y|^a\,,\qquad &\textup{if $0<a\leq 1\,$,}\\
2^{a-1}|x|^a+2^{a-1}|y|^a\,,\qquad &\textup{if $a>1\,$.}
\end{cases}
$$

By $\cP(\Rd)$ we denote the space of the probability measures, by $\cP_a (\Rd)$ -- the space of probability measures with finite $a$th moments such that $\cP_a (\Rd) = \{ \mu \in \cP ; \mu(|\cdot|^a) < +\infty \}$, while by $\cP_{\rm ac}$ absolutely continuous ones (having a density with respect to the Lebesgue measure). At each measure $\mu \in \cP_2 (\Rd)$ we consider a tangent space (bundle) defined as 
\begin{flalign*}
    {\rm Tan}_{\mu} (\cP_2 (\Rd)) &:= \overline{ \{\nabla \phi : \phi \in C_c^{\infty} (\Rd) \} }^{L^2 (\mu)} \\ &= \left\{ v \in L^2 (\mu);\ \int_{\R^d} v \cdot w \, \d \mu = 0 , \ \forall w \in L^2 (\mu) \ \mbox{such that} \ \dv (w\mu) = 0 \right\} \, . 
\end{flalign*}
Thus, we consider the space $\cP_2 (\Rd)$ as a~weighted space where we assign to each element $\mu$ a Hilbert space $L^2 (\mu)$. Elements of the tangent bundle ($L^2 (\mu)$ space) are vector fields. They are gradients of a~compactly supported smooth functions on $\Rd$, i.e., $C_c^\infty (\Rd)$.

We consider the metric space $(\cP_2 (\Rd), W_2)$ endowed with the $2$-Wasserstein distance. The $2$-Wasser\-stein distance between two probability measures $\mu^1,\mu^2\in \cP_2(\R^d)$ is deﬁned by
\begin{equation} \label{eq:w2_dist}
W_2 (\mu^1,\mu^2) := \left(\inf\left\{
\int_{\R^d\times\R^d}|x_1 - x_2 |^2 \,\d \boldsymbol \mu(x_1 , x_2 );\ \boldsymbol\mu\in\Gamma(\mu^1,\mu^2) \right\}\right)^\frac{1}{2}\,,
\end{equation}
or equivalently
$$
W_2 (\mu^1, \mu^2) :=\left(\inf\left\{\Ex \lvert X - Y \rvert^2;\ X \sim \mu^1 , Y \sim \mu^2 \right\}\right)^{\frac{1}{2}}\,.
$$
We denote $\Gamma(\mu^1, \mu^2, \ldots , \mu^N)$ as a set of \textit{admissible plans} such that
\begin{equation*}
\Gamma(\mu^1, \ldots , \mu^N) := \{ \boldsymbol\mu \in \cP (X_1 \times \cdots \times X_N); \ \Pi^i_\# \boldsymbol\mu = \mu^i; \ i = 1, \ldots , N \} \, .
\end{equation*}
Here, $\Pi^i(X_1\times\cdots\times X_N)=X_i$ and $\Pi^i_\# \boldsymbol\mu$ denotes the pushforward of the vector $\boldsymbol\mu$, i.e., for any measurable $f:\Rd\to\R^p$, for $p\geq 1$, and any measure $\nu\in\cP(\Rd)$ we denote \begin{equation}\label{eq:pushforward}
f_\#\nu (U):=\nu((f)^{-1}(U))\qquad\textup{for any measurable set $U\subseteq\R^p$}\,.
\end{equation} 
Especially for $N=2$ we call measure $\boldsymbol\mu \in \Gamma(\mu^1, \mu^2)$ a \textit{transport plan.} We call the  transport plan $\boldsymbol\mu \in \Gamma_o (\mu^1, \mu^2)$ \textit{optimal} iff it attains the~infimum of \eqref{eq:w2_dist}.
In general, $X$ can be a~Hilbert space. In the paper, we consider $X=\Rd$. During the text we will consider extended real functionals from the space of probability measures $\cG: \cP(\Rd) \to (- \infty, +\infty ]$ with its proper domain such that 
\begin{equation*}
\operatorname{dom} \cG := \{ \nu \in \cP(\Rd) ; \ \cG [\nu] < + \infty \} \neq \emptyset \, .
\end{equation*}

By the {\it proximal operator} of the function $V$ with the time step $\tau$ we mean an operator $\prox_V^\tau : \R^d \rightarrow \R^d$ given by 
\begin{equation}\label{def:proximal}
    \proxv (x) := \argmin_{y\in\R^d} \left\{ V(y) + \tfrac{1}{2\tau} \lvert y-x \rvert^2\right\} \, .
\end{equation} 
In discrete time $k=1,\ldots,n$ with step $\tau$, we say that $\wt{\vr}_{k}$ is a~solution of the {\it Jordan--Kinderlehrer--Otto scheme} ({\it JKO scheme} for short) for a functional $\cG: \cP(\Rd) \to (- \infty, +\infty ]$  if
\begin{equation*}
\begin{cases}
\wt{\vr}_{k}\in \underset{\nu\in\cP_{\rm ac}}{\rm arg\,min} \left\{ \cG[\nu] + \tfrac{1}{2\tau} W_2^2 (\nu, \wt{\vr}_{k-1} )\right\}\,,\\
\wt\vr_0=\vr_0(x)\,.
\end{cases}
\end{equation*} 

Pretty complete exposition of the theory of the JKO scheme can be found in \cite[Chapter~4]{Ambrosio-Gigli-Savare}. For general data it is known that $W_2^2(\wt\vr_k,\mu(t))\leq C\tau\,$, which can be improved to $W_2^2(\wt\vr_k,\mu(t))\leq C\tau^2\,$ under additional assumptions, see \cite[Theorem~4.0.10]{Ambrosio-Gigli-Savare}. 

\medskip

\noindent{\bf Gradient flows. } Let us comment on the basic properties of~\eqref{eq:model} and introduce the main tools needed for its analysis. Our main reference for this part is \cite{Ambrosio-Gigli-Savare}, but one can get also insight from \cite{Santambrogio-overview,Figalli-Glaudo}. 

A {\it generalized geodesic} joining $\mu^2$ to $\mu^3$ (with base $\mu^1$) is a~curve of the type
\[
\mu_\theta^{2\rightarrow 3} = ((1-\theta)\mu^2 + \theta \mu^3)_\# \boldsymbol{\mu}; \quad \theta \in [0,1],
\]
where 
\[
\boldsymbol{\mu} \in \Gamma (\mu^1, \mu^2, \mu^3)\,, \quad  \Pi^{1,2}_\# \boldsymbol{\mu} \in \Gamma_o (\mu^1,\mu^2)\,,\quad \text{and} \quad \Pi^{1,3}_\# \boldsymbol{\mu} \in \Gamma_o (\mu^1, \mu^3) \, .
\]
Given $\lambda\in\R$, a~functional $\cG:\mathscr{P}_2(\R^d) \rightarrow  \R \cup \{+\infty\}$ is called $\lambda$-{\it convex along generalized geodesics} if for any $\mu^1, \mu^2, \mu^3 \in \operatorname{dom} \cG$ there exists an generalized geodesics $\mu_\theta^{2\rightarrow 3}$  induced by a~plan $\boldsymbol{\mu} \in \Gamma_o  (\mu^1, \mu^2, \mu^3)$ such that
\[
\cG[\mu_\theta^{2\rightarrow 3}]  \leq \theta \cG [\mu^2]+ (1-\theta)\cG[\mu^3] - \tfrac{\lambda}{2} \theta (1-\theta) W_{2,\boldsymbol{\mu}}^2 (\mu^2, \mu^3)\quad \forall \theta \in [0,1] \, ,
\]
where $W_{2,\boldsymbol{\mu}}^2 (\mu^2, \mu^3)$ is defined by 
\begin{equation*}
   W_{2,\boldsymbol{\mu}}^2 (\mu^2, \mu^3) := \int |x_3-x_2|^2 \, \mathrm{d}\boldsymbol{\mu}(x_1,x_2,x_3) \geq W_2^2 (\mu^2, \mu^3) \, . 
\end{equation*} 
A {\it gradient flow} of a functional $\cG$ in the Wasserstein space $(\mathscr{P}_2 (\R^d), W_2)$ is defined as 
\begin{equation} \label{eq:gf}
 - v_t \in  \partial \cG [\mu_t]\, ; \ v_t \in \mathrm{Tan}_{\mu_t} (\mathscr{P} (\R^d)) 
\end{equation}
according to \cite[Definition 11.1.1]{Ambrosio-Gigli-Savare}. It can be characterised by the Evolution Variational Inequalities (EVI for short).

\begin{prop} [EVI, {\cite[Theorem~11.1.4]{Ambrosio-Gigli-Savare}}] 
Let the functional $\cG : \cP_2 (\R^d) \to \R$ be proper, lower semicontinuous, coercive, and $\lambda$-convex along generalised geodesics with $\lambda \geq 0$. Then the gradient flow \eqref{eq:gf} is characterised by the Evolution Variational Inequality such that 
\begin{equation*} 
\tfrac{1}{2} \tfrac{\mathrm{d}}{\mathrm{d}t} W_2^2 (\mu_t, \nu) \leq \cG [\nu] - \cG [\mu_t] - \tfrac{\lambda}{2} W_2^2 (\mu_t, \nu) \quad \forall \nu \in \mathrm{dom}\, \cG \, .
\end{equation*}
\end{prop}

\subsection{On the existence and uniqueness of solutions} \label{ssec:ex-n-uni}
We shall identify solutions to the nonlinear Fokker--Planck equations~\eqref{eq:model} and densities of solutions to McKean--Vlasov equation~\eqref{eq:mckean_vlasov:model}. 

We call a solution to the SDE {\it strong} if the integrated equation holds almost surely. Moreover, if we have a strong solution $Y_t$ to~\eqref{eq:mckean_vlasov:model} with density $\mu_t$, which is absolutely continuous with respect to the Lebesgue measure given an initial moment condition due to \cite[Page 29]{CGM}, then, by the Itô formula, we can show that $\mu_t$ satisfies~\eqref{eq:model} in the sense of distributions. Indeed, let us take $\phi\in \mathscr{D}(\R^d)$, $Y_t$ satisfying~\eqref{eq:mckean_vlasov:model}, and apply  Itô formula to $\phi(Y_t)$ to get
\begin{flalign*}
    \phi(Y_t)=\phi(Y_0)-\int_{0}^t\bigg( \nabla \phi(Y_s)\cdot\Big(\nabla V(Y_s)+\big(\nabla W*\mu_s\big)(Y_s)\Big)-\Delta\phi(Y_s)\bigg)\ds+\int_0^t\sqrt{2}\nabla\phi(Y_s)\cdot\d B_s\,.
\end{flalign*}
Taking the expectation on both sides above, we get \begin{flalign*}
    \int_{\R^d}\phi(y)\d\mu_t(y)=\int_{\R^d}\phi(y)\d\mu_0(y)-\int_{0}^t \int_{\R^d}\bigg(\nabla \phi(y)\cdot\Big(\nabla V(y)+\big(\nabla W*\mu_s\big)(Y_s)\Big)-\Delta\phi(y)\bigg)\d\mu_s(y)\ds \,
\end{flalign*}
and $\mu(\cdot,t)=\mu_t(\cdot)$ is a distributional solution to~\eqref{eq:model}.  

\medskip
 
\noindent Let us comment on the known results on the existence and uniqueness to both abovementioned problems.

\medskip

\noindent{\bf SDE~\eqref{eq:mckean_vlasov:model}. } We refer to~\cite{oksendal2010stochastic,EK} for the basics of the theory of SDEs. 
The existence theory for the McKean--Vlasov equation with nonlocal interactions is an active area of research in mathematical analysis and probability theory. The interaction potential $W$ is often assumed to be locally Lipschitz continuous, bounded, or to satisfy suitable decay or growth conditions. The initial data $\mu_0$ is typically assumed to be nonnegative and integrable to ensure that the solution remains well-defined and nonnegative for all time.  Assumptions on the existence of moments of the initial distribution function $\mu_0$ may be required to ensure the well-posedness of the equation, and as we have seen above it is also required to relate the (absolutely continuous) density coming from the SDE with the solution of the corresponding PDE. In our approach the well-posedness of the SDE is thus essential, but it is not our main focus and we rely on previous theory. 

The study of existence of~\eqref{eq:mckean_vlasov:model} dates back to~\cite{McKean}. See also~\cite{Sznitman,Meleard} for a study on the no interaction-model, i.e., when $W\equiv 0$. The mentioned contributions deal with the existence to problems with Lipschitz coefficients. We refer to~\cite[Theorem~2.6]{CGM} for existence and uniqueness to~\eqref{eq:mckean_vlasov:model} under $m$-polynomial growth conditions from above imposed on $\nabla W$ and $\nabla V$. The assumptions therein are satisfied under our \ref{V} and \ref{W} for $m=\max\{q_V,q_W\}\geq 1$. Note, however, that the integrability of data required in~\cite[Theorem~2.6]{CGM} depends on $m$, i.e., $\mu_0\in\cP_2(\R^d)\cap \cP_{\max\{m(m + 3), 2m^2\}}(\R^d)$. We point out that \cite[Theorem 2.13]{HerrmannImkellerPeithmann2008} provides the existence and uniqueness of the SDE including moment bounds, under a proper regime including locally Lipschitz behaviour of $V,W$  and polynomial growth of $W$ itself.  In \cite[Proposition 1]{surv-I}, it is proven that the PDE and SDE are equivalent in the globally Lipschitz case.
 
\medskip

\noindent{\bf PDE~\eqref{eq:model}. } Our main reference for a PDE approach is \cite{CMCV}. Problem~\eqref{eq:model} might be effectively studied with the use of the gradient flow theory, see \cite{Ambrosio-Gigli-Savare,Figalli-Glaudo}. The relation between \eqref{eq:model} and minimization of $\cF$ given by~\eqref{cF} is already exposed in the introduction. 

The typical assumptions on the  interaction potential $W$ such as smoothness, convexity, and decay properties, help ensure the well-posedness of the equation and the existence of solutions. Additional specific assumptions on the interaction potentials may be necessary to capture physical phenomena such as repulsion or attraction between particles. The    regularity of the initial datum is imposed in order to preserve physical constraints such as mass conservation.  The friction term (involving $V$) is often assumed to be convex or to satisfy certain growth conditions. It is also done to reflect physically realistic behavior. For example, convexity implies that the frictional force increases with the relative velocity between particles, which aligns with our intuition about how friction behaves in real granular materials. On the other hand, these assumptions lead to the existence, uniqueness, and stability of solutions under perturbations.

We note that by \cite[Chapter~11]{Ambrosio-Gigli-Savare}, we can ensure existence and uniqueness of solutions to~\eqref{eq:model}. In fact, what is needed there is:
\begin{itemize}
    \item $V$ is lower semi-continuous, $\lambda_V$-convex, and having proper domain with nonempty interior $\Omega\subset\R^d$
    \item $W$ is convex, differentiable, even, and doubling in the sense that there exists $c_W$ such that for every $x,y$ it holds $W(x+y)\leq c_W(1+W(x)+W(y))$.
\end{itemize}
We note that under these assumptions \cite[Theorem~11.2.8]{Ambrosio-Gigli-Savare} yields existence of a unique distributional solution $\mu_t\in L^1_{\textup{loc}}((0,+\infty);W^{1,1}_{\textup{loc}}(\R^d))$. Moreover, the obtained unique distributional solution $\mu_t\in\cP_2(\r^d)$ is a gradient flow of~\eqref{cF}, which preserves mass along the flow and is continuous (in $\cP_2(\r^d)$) as $t\to 0$. Let us stress that there is no need to impose any particular growth of the fashion of \cite[Section~2]{CMCV} or \cite[Section~6]{CMCV-ARMA}  for the existence itself. 

Despite the fact that the  assumptions on $W$, that we have just mentioned, do not completely overlap with \ref{W} under which we prove the main results, we decided to give a conditional result (assuming the existence of gradient flow solutions).  Finding exact references in our case is not the main focus of the paper. However, the existence of a distributional solution is guaranteed from the SDE approach explained above. Whether this solution is a gradient flow or not is an open problem in our full generality.

\section{Propagation of chaos to handle nonlocal interactions}\label{sec:nonlocal-problem}

As already mentioned, assuming existence of higher moments of the initial distribution function $\mu_0$ is typical to ensure the well-posedness of the equations \eqref{eq:model} and \eqref{eq:mckean_vlasov:model}. It is even more expected in the study of propagation of chaos. We will now complete and extend the proofs in \cite[Theorem 3.3]{MALRIEU_log_sobolev}, \cite[Section 3.1.3]{surv-II}, and \cite{CGM} to obtain improved estimates on moment bounds and propagation of chaos in our non-globally Lipschitz setting. See also the related Theorem 2.3 and Proposition 2.5 in \cite{ChDR24}.

\begin{theo}[Moments bound]\label{thm:cont-moments}
We assume that $V : \mathbb{R}^d \to \mathbb{R}$ attains its minimum at $x^*=0$ and is $\lambda_V$-convex for $\lambda_V>0$, and $W: \mathbb{R}^d \to \mathbb{R}$ is convex and radially symmetric. Let $\Yt$ be a solution to~\eqref{eq:mckean_vlasov:model} for $Y_0\sim\mu_0\in\cP_2(\R^d)\cap\cP_a(\R^d)$ for $a\geq 0$. Then 
\begin{equation*}
    \sup_{t \geq 0} \Ex \lvert Y_t \rvert^a \leq \max\left\{ \Ex\lvert Y_0\rvert^a,\left(\frac{d+(a-2)_+}{\lambda_V}\right)^{\tfrac{a}{2}}\right\}\,.
\end{equation*}
\end{theo}

\begin{rem}\rm 
If the minimum of $V$ was not at $x^*=0$, the result would have been
\begin{equation*}
    \sup_{t \geq 0} \Ex \lvert Y_t-x^* \rvert^a \leq \max\left\{ \Ex\lvert Y_0 -x^*\rvert^a,\left(\frac{d+(a-2)_+}{\lambda_V}\right)^{\tfrac{a}{2}}\right\}\,.
\end{equation*}
\end{rem}

\begin{proof}
We begin with the case $a\geq 2$. Apply the Itô formula for $\Yt\mapsto \lvert \Yt \rvert^a$ to get that 
\begin{multline*}
\d \lvert \Yt \rvert^a = - a \lvert \Yt \rvert^{a-2} \Yt \cdot \big(\nabla V(\Yt) + (\nabla W * \mut)(\Yt) \big)\dt + a(  d + a-2)\lvert \Yt \rvert^{a-2}  \dt + a d  \sqrt{2 } \lvert \Yt \rvert^{a-2} \Yt  \d \Bt \, .
\end{multline*}
We take $\tYt$ as an independent copy of $\Yt$ and use that for $X \sim \mu$ it holds $(\nabla W * \mu )(\eta) = \Ex [\nabla W(\eta - X) ]$. Then
\begin{flalign*}
\d \lvert \Yt \rvert^a = &- a \lvert \Yt \rvert^{a-2} \Yt \cdot \nabla V(\Yt) \dt - a \lvert \Yt \rvert^{a-2} \Yt \cdot \Ex [\nabla W(\Yt - \tYt) |\Yt] \dt \\ &+ a(  d + a-2)\lvert \Yt \rvert^{a-2}  \dt + a d \sqrt{2 } |Y_t|^{a-2} Y_t\d \Bt \, .
\end{flalign*}
Now, we take the expectation on the both sides and consider the property of the Brownian motion such that
\begin{flalign}
    \notag
\d \Ex \lvert \Yt \rvert^a =& - a \, \Ex [ \lvert \Yt \rvert^{a-2} \Yt \cdot \nabla V(\Yt) ] \dt - a \, \Ex [\lvert \Yt \rvert^{a-2} \Yt\cdot \nabla W(\Yt - \tYt)] \dt\\& + a(  d + a-2) \Ex [ \lvert \Yt \rvert^{a-2}]\dt \, .\label{ODE}
\end{flalign}
Our aim now is to estimate the terms involving $V$ and $W$. Since $V$ is $\lambda_V$-convex and $\nabla V(0) = 0$,  we have 
\begin{equation*} 
x \cdot \nabla V(x) \geq \lambda_V \lvert x \rvert^2 \quad \forall x \in \Rd\,.
\end{equation*} Note that
\begin{equation*}
\Ex  [\lvert \Yt \rvert^{a-2} \Yt \cdot \nabla V (\Yt)] \geq {\lambda_V} \Ex  [\lvert \Yt \rvert^a]\, .
\end{equation*}

Now, we will show the following claim 
\begin{equation} \label{eq:W_ex}
\Ex  \left[\lvert \Yt \rvert^{a-2} \, \Yt \cdot \nabla W(\Yt - \tYt)\right]  \geq 0 \, .
\end{equation}
Due to the Fubini theorem, the fact that $W$ is even (by radial symmetry), and so $\nabla W$ is odd, we can write that
\[
\Ex  \big[\lvert\Yt\rvert^{a-2}\,\Yt \cdot \nabla W (\Yt - \tYt)\big] = \Ex  \big[\lvert\tYt\rvert^{a-2}\,\tYt \cdot \nabla W (\tYt - \Yt)\big] = - \Ex  \big[\lvert\tYt\rvert^{a-2}\,\tYt \cdot \nabla W (\Yt - \tYt)\big] \,. \]
Consequently, \[\Ex \big[\lvert\Yt\rvert^{a-2}\,\Yt \cdot \nabla W (\Yt - \tYt)\big] = \tfrac{1}{2} \Ex  \big[(\lvert\Yt\rvert^{a-2}\,\Yt - \lvert\tYt\rvert^{a-2}\,\tYt) \cdot \nabla W(\Yt - \tYt)\big] \, .
\]
To show that this term is greater or equal than zero, let us consider the function $G\colon\R\to\R$ such that 
\begin{equation*}
G(t)=W(y-x+t(\lvert y\rvert^{a-2}y-\lvert x\rvert^{a-2}x)) \, .
\end{equation*}
We see that 
\begin{equation*}
G'(t) = \nabla  W(y-x+t(\lvert y\rvert^{a-2}y-\lvert x\rvert^{a-2}x)) \cdot (\lvert y\rvert^{a-2}y-\lvert x\rvert^{a-2}x) \, .
\end{equation*}
By existence of $\nabla W$ for all $x$ in $\Rd$, the derivative $G'(t)$ exists for all $t$ in $\R$ and particularly the quantity $G'(0)$ is the one we want to estimate.
By convexity assumption and radial symmetry, $\lvert w \rvert \geq \lvert z \rvert$ iff $W(w) \geq W(z)$ for any $w, z \in \Rd$. This implies $G'(t) \geq 0$ for any $t \geq 0$. Thus we need to show that 
\begin{equation*}
\big\lvert y-x+t(\lvert y\rvert^{a-2}y-\lvert x\rvert^{a-2}x) \big\rvert \geq \big\lvert y-x\big\rvert \qquad \forall x,y\in \Rd \ \text{and} \ t \geq 0 \, .
\end{equation*}
Since both sides are nonnegative, by taking square we get 
\begin{equation*}
\big\lvert y-x+t(\lvert y\rvert^{a-2}y-\lvert x\rvert^{a-2}x) \big\rvert^2 = \lvert y-x\rvert^2 + 2t \, (y-x) \cdot (\lvert y\rvert^{a-2}y-\lvert x\rvert^{a-2}x) + t^2 \big\lvert \lvert y\rvert^{a-2}y-\lvert x\rvert^{a-2}x \big\rvert^2 \, .
\end{equation*}
It remains to show that the middle term is nonnegative. We rearrange the term and use the Cauchy-Schwarz inequality $u\cdot v \leq \lvert u \rvert \, \lvert v \rvert$ for all $u$, $v$ from $\Rd$ such that 
\begin{flalign*}
(y-x) \cdot (\lvert y\rvert^{a-2}y-\lvert x\rvert^{a-2}x) &= \lvert y \rvert^a - \lvert x \rvert^{a-2} \, y\cdot x - \lvert y \rvert^{a-2} \, x\cdot y + \lvert x \rvert^{a} \\
&\geq \lvert y \rvert^a  - \lvert x \rvert^{a-1} \, \lvert y \rvert - \lvert y \rvert^{a-1} \, \lvert x \rvert + \lvert x \rvert^{a-1} = (\lvert y \rvert^{a-1} - \lvert x \rvert^{a-1}) \, (\lvert y \rvert - \lvert x \rvert) \geq 0 \, .
\end{flalign*}
The last inequality holds since the function $|\cdot|^{a-1}$ is monotone. Thus, we have shown the claim \eqref{eq:W_ex}.

We denote  $m_a(t) := \Ex \lvert \Yt \rvert^a $ and plug the identity from the last display into \eqref{ODE}. Summarizing the above observations  for any $a\geq 0$ we get
\begin{equation*}
m'_a (t) \leq a\left[ -{\lambda_V}m_a(t) + (d +a-2)  m_{a-2}(t)\right]\, .
\end{equation*}
Note that $m_a$ is decreasing whenever $m_a(t)\geq (d +a-2)  m_{a-2}(t)/\lambda_V$. 
Moreover, since the Jensen inequality implies that $m_{a-2}(t)\leq m_{a}^{(a-2)/a}(t)$, we find that 
\begin{flalign*}
 \tfrac{d +a-2}{\lambda_V}  m_{a-2}(t)\leq  \tfrac{d+a-2}{\lambda_V}m_{a}^{(a-2)/a}(t)\leq m_a(t)\,
\end{flalign*}
as long as $m_a(t)\geq ((d+a-2)/\lambda_V)^\frac{a}{2}$. By summing up the above estimates, we conclude with the claim.

In the case $0<a< 2$, we apply the Jensen inequality $\Ex|X_t|^a\leq \big(\Ex|X_t|^2\big)^{\frac{a}{2}}$.
\end{proof}

Before presenting our result on propagation of chaos, we remind the reader that we denote $\mu_t^N=\mu_t^{i,N}$ for any $i$ where $\boldsymbol{X}_t\sim \boldsymbol{\mu}_t=(\mu_t^{1,N},\ldots,\mu_t^{N,N})$.

\begin{theo}[Propagation of chaos]\label{theo:chaos}
We assume $V: \R^d \to \R$ attains its minimum at $x^*=0$ and is $\lambda_V$-convex for some $\lambda_V>0$, and $W: \R^d \to \R$ satisfies \ref{W} with some $q_W>1$ and $L_{q_W}>0$. Let $\mu$ be a solution to~\eqref{eq:model} for $\mu_0\in\cP_2(\R^d)\cap\cP_{2(q_W-1)}(\R^d)$. Then there exists constant $C=C(q_W,L_{q_W},\lambda_V,\mu_0,x^*)>0$ such that, for every $N \geq 1$, it holds
\begin{equation*}
\sup_{t \geq 0} W_2^2(\mu^N(t),\mu(t)) \leq C \frac{d^{2(q_W-1)}\lambda_V^{-2q_W}}{ N} \,.
\end{equation*}
Moreover, $C\ll +\infty$ when $\lambda_V \to 0$.
\end{theo}

\begin{proof}
Let $Y_t^i$ denote the solution of \eqref{eq:mckean_vlasov:model} where the Brownian motion $B_t$ is replaced with $B_t^i$ from \eqref{eq:particle_system} such that 
\begin{equation} \label{eq:mckean_vlasov_inter_i}
\d \Yti = -\left(\nabla V(\Yti) + (\nabla W * \mu(t)) (\Yri) \right)\dt + \sqrt{2} \d B_t^i \qquad \text{for }\ i = 1, \ldots , N \, .
\end{equation}
We subtract \eqref{eq:particle_system}~and \eqref{eq:mckean_vlasov_inter_i} and taking into account that Brownian motions in the equations are correlated we obtain  
\begin{equation*}
\Xti - \Yti = \Xsi - \Ysi - \int_s^t \big(\nabla V(\Xri) - \nabla V(\Yri)\big) \dr 
- \frac{1}{N} \sum_{\substack{j=1 \\ i \neq j}}^{N}  \int_s^t \big(\nabla W (\Xri - \Xrj) - (\nabla W * \mu(r)) (\Yri) \big) \dr \, .
\end{equation*}
Let us define\begin{equation*}
    \zeta_{i,j}(r):= (\Xri - \Yri)\cdot\big(\nabla W (\Xri - \Xrj) - (\nabla W * \mu(r)) (\Yri) \big)\,.
\end{equation*}
 Now we consider the transformation $\Xti - \Yti \mapsto \lvert \Xti - \Yti \rvert^2$. The Itô lemma implies that for $ i = 1, \ \ldots , N$ it holds
\begin{equation} \label{eq:ito_one_line}
\lvert \Xti - \Yti \rvert^2 = \lvert \Xsi - \Ysi \rvert^2 
- 2 \int_s^t (\Xri - \Yri) \cdot \big(\nabla V(\Xri) - \nabla V(\Yri)\big) \dr - \frac{2}{N} \sum_{\substack{j=1 \\ i \neq j}}^{N}  \int_s^t \zeta_{i,j}(r) \dr \,.
\end{equation}
We split $\zeta_{i,j}$ and define 
\begin{multline*}
\zeta_{i,j} (r) = (\Xri - \Yri)\cdot \left[\nabla W (\Xri - \Xrj) -  \nabla W(\Yri - \Yrj) \right] + (\Xri - \Yri) \cdot \left[ \nabla W(\Yri - \Yrj) - (\nabla W * \mu(r)) (\Yri) \right] \\=: \vp_{i,j} (r) + \psi_{i,j} (r) \, .
\end{multline*}
Since $\nabla W$ is odd, we find that
\begin{align*}
    \vp_{i,j}(r) + \vp_{j,i}(r)  &= 
    (\Xri - \Yri)\cdot \left[\nabla W (\Xri - \Xrj) -  \nabla W(\Yri - \Yrj) \right] \\ 
    &\qquad + (\Xrj - \Yrj)\cdot \left[-\nabla W (\Xri - \Xrj) +  \nabla W(\Yri - \Yrj) \right]  \\
    &= \left[ (\Xri - \Xrj) - (\Yri - \Yrj) \right] \cdot  \left[ \nabla W (\Xri - \Xrj) -  \nabla W(\Yri - \Yrj) \right] \geq 0 \, ,
\end{align*}
where the last inequality follows directly from the convexity of $W$. Consequently
\begin{equation*}
\sum_{\substack{j,i=1, \, i > j}}^{N}  \big(\vp_{i,j}(r) + \vp_{j,i}(r) \big) \geq 0 \, .
\end{equation*}
We denote 
\begin{equation*}
\vt_{i,j} (r):= \nabla W(\Yri - \Yrj) - (\nabla W * \mu(r)) (\Yri) 
\end{equation*}
Then the Cauchy--Schwarz inequality such that 
\begin{flalign*}
\left| \Ex \left[ \sum_{j=1}^{N} \psi_{i,j} (r) \right]\right| &= \left| \Ex \left[ (\Xri - \Yri) \cdot  \vtij (r) \right]\right| \leq 
\sqrt{\Ex \left[ \left\lvert \Xri - \Yri  \right\rvert^2 \right]}  \sqrt{ \Ex \left[ \left\lvert \Sigma_{j=1}^{N} \vtij \right \rvert^2 \right] } \\ &= \sqrt{\Ex \left[ \left\lvert \Xri - \Yri  \right\rvert^2 \right]} 
 \sqrt{ \sum_{j=1}^{N} \Ex [\lvert \vtij (r) \rvert^2] + 2 \sum_{\substack{j,k=1, \, j < k}}^{N}  \Ex [\vtij (r) \cdot \vtik(r)]} \, .
\end{flalign*}
We consider the last sum on the right-hand side above. The condition $j < k$ implies that at least one index of $j$ and $k$ is not equal to $i$. Then
\begin{align*}
\Ex [\vtij (r) \cdot \vtik(r)|\Yri] &=\Ex  [ \nabla W(\Yri - \Yrj) \cdot \nabla W(\Yri - \Yrk)|\Yri] \\
&\quad- (\nabla W * \mu(r)) (\Yri) \cdot \Ex [\nabla W(\Yri - \Yrj) |\Yri] \\  
&\quad -  (\nabla W * \mu(r)) (\Yri)\cdot \Ex  [ \nabla W(\Yri - \Yrk)|\Yri ]  + \lvert (\nabla W * \mu(r)) (\Yri) \rvert^2\,.
\end{align*}
Since $\Yrj$ and $\Yrk$ are independent copies of $\Yr^{1} \sim \mu(r)$, we infer that
\begin{align*}
\Ex [\vtij (r) \cdot \vtik(r)|\Yri] 
=& \Ex[ \nabla W(\Yri - \Yrj)|\Yri] \cdot \Ex [ \nabla W(\Yri - \Yrk)|\Yri] \\
&-(\nabla W * \mu(r)) (\Yri) \cdot \Ex [\nabla W(\Yri - \Yrj)|\Yri ] \\
&-   (\nabla W * \mu(r)) (\Yri) \cdot  \Ex[\nabla W(\Yri - \Yrk)|\Yri] + \lvert (\nabla W * \mu(r)) (\Yri) \rvert^2 \, .
\end{align*}
Taking into account that $\Ex[\nabla W (\Yri - \Yrj)|\Yri] = (\nabla W * \mu(r)) (\Yri)$ and the fact that each $\Yrj$ has the same law $\mu(r)$ and  $\Ex [\nabla W(\Yri - \Yrj) |\Yri] = \Ex [\nabla W(\Yri - \Yrk)|\Yri]$, we get 
\begin{equation*}
\Ex [\vtij (r) \cdot \vtik (r)|\Yri] =0 \qquad \mbox{whenever} \qquad j \neq k \, 
\end{equation*}
which implies
\begin{equation*}
\left| \Ex \left[ \sum_{j=1}^{N} \psi_{i,j} (r) \right]\right| \leq \sqrt{\Ex \left[ \left\lvert \Xri - \Yri  \right\rvert^2 \right]} \cdot \sqrt{ \sum_{j=1}^{N} \Ex [\lvert \vtij (r) \rvert^2]} \, .
\end{equation*}
We have
\begin{align*}
\Ex \left[ \lvert \vtij (r) \rvert^2 \right] &= \Ex \left[ \lvert \nabla W(\Yri - \Yrj) - (\nabla W * \mu(r)) (\Yri) \rvert^2 \right] \\
&= \Ex \left[ \lvert \nabla W(\Yri - \Yrj) \rvert^2 \right] - \lvert  (\nabla W * \mu(r)) (\Yri) \rvert^2 \leq \Ex\left[ \lvert \nabla W(\Yri - \Yrj) \rvert^2 \right] \,.
\end{align*}
Due to \ref{W}, $\nabla W$ is $q_W$-power growth Lipschitz, we can then apply Theorem \ref{thm:cont-moments} to get
\begin{align*}
\Ex \lvert \nabla W(\Yri - \Yrj) \rvert^2  &\leq 4L_{q_W}^2 \left(1+\Ex\lvert \Yrj \rvert^{2q_W-2} +  \Ex \lvert \Yrk \rvert^{2q_W-2} \right)\\
&\leq  4L_{q_W}^2\left(1+2\left(\frac{d+2(q_W-2)_+}{\lambda_V}\right)^{q_W-1}\right)\\
&\leq 4L_{q_W}^2\left(1+\max\{2,2^{q_W-1}\}\frac{d^{q_W-1}}{\lambda_V^{q_W-1}}+\max\{2^{q_W},2^{2q_W-2}\}\frac{(q_W-2)_+^{q_W-1}}{\lambda_V^{q_W-1}}\right)=:\mathcal{M}\, ,
\end{align*}

where we, without loss of generality, assume $Y_0$ is such that $\Ex|Y_0|^{2(q_W-1)}\leq (\tfrac{d+2(q_W-2)_+}{\lambda_V})^{q_W-1}$ and we let $\mathcal{M}=\mathcal{M}(q_W,L_{q_W},\lambda_V,\mu_0)$. Thus we have 
\begin{equation*}
\left| \Ex \left[ \sum_{j=1}^{N} \psi_{i,j} (r) \right]\right| \leq \sqrt{\Ex  \left\lvert \Xri - \Yri  \right\rvert^2 } \sqrt{N\mathcal{ M}}\,.
\end{equation*}
We denote 
\begin{equation*}
\alpha (t): = \Ex \left\lvert \Xti - \Yti  \right\rvert^2  \, .
\end{equation*}
From \eqref{eq:ito_one_line} and $\lambda_V$-convexity of $V$ we infer that 
\begin{equation*}
\alpha (t) \leq \alpha (s) -2 \lambda_V \int_s^t \alpha (r) \dr + 2\frac{ \mathcal{M}}{\sqrt{N}} \int_s^t \sqrt{\alpha(r)} \dr \, .
\end{equation*}
This means that 
\begin{equation*}
\alpha^\prime (t) \leq -2 \lambda_V \alpha (t) + 2\tfrac{ \mathcal{M}}{\sqrt{N}} \sqrt{\alpha (t)},
\end{equation*}
and by the Grönwall lemma 
we get 
\begin{equation*}
\sqrt{\alpha (t)} \leq \tfrac{\mathcal{M}}{\lambda_V \sqrt{N}} (1- \mathrm{e}^{-\lambda_V t}) \, .\qedhere
\end{equation*}
\end{proof}

\begin{rem}[Relaxing $\lambda$-convexity]\label{rem:relax-lambda-conv} \rm 
Problems involving $\lambda \in\R$ are studied e.g. by~\cite{carrillo_prop_chaos_lambda,durmus_propagation}. In fact, \cite[Theorem~2]{carrillo_prop_chaos_lambda} provides propagation of chaos for $\lambda\in\R$ under extra assumptions ($V,W$ are bounded below and $W$ is doubling), but they do not specify the rate of convergence. On the other hand, the assumptions of \cite{durmus_propagation} are less comparable to our regime. In \cite[Corollary~3]{durmus_propagation}, propagation of chaos is shown in the $1$-Wasserstein distance.  We point out that \cite{durmus_propagation} captures the following natural example: $V(x)=|x|^4-a|x|^2$, $a>0$, and $W(x)=\pm |x|^2$. For more information see \cite{surv-I,surv-II}.
\end{rem}

\section{Numerical approximation of the local problem}\label{sec:local-problem}
 In this section we first prove Theorem~\ref{thm:alg_bound-intro} yielding that the Algorithm~\ref{alg:dgf} approximates the gradient flow solutions of \eqref{eq:model} with $W\equiv0$. Then, as a consequence, we show the proof of Theorem~\ref{thm:alg_bound-nonlocal}, i.e., we motivate that Algorithm~\ref{alg:particle}  approximates \eqref{eq:model} with $W\not\equiv0$.

Alternating steps of Algorithm~\ref{alg:dgf} go along gradient flows of $\cFV$ and $\cFE$, respectively, which is provided below.
\begin{lem} 
Let $V$ be convex. Step {\rm (1)} of Algorithm~\ref{alg:dgf} is a step of length $\tau$ along the gradient flow of the potential functional~$\cFV$.
\end{lem}

\begin{proof}
    We mimic the steps in the proof of \cite[Proposition 10.4.2{\it (ii)}]{Ambrosio-Gigli-Savare}.
    For every $\nu \in \operatorname{dom} \cF_V$ and $\boldsymbol{\gamma} \in \Gamma_o (\nu, \vrk)$ we have 
    \begin{align*}
    \cFV[\nu] + \frac{1}{2 \tau} W_2^2(\nu, \vrk) &= 
    \iint\left( V(x_2) + \frac{1}{2 \tau} \lvert x_2 - x_1 \rvert^2\right) \, \mathrm{d} \boldsymbol{\gamma} \\
    &\geq \int\left( V(\prox_V^\tau (x_1)) + \frac{1}{2 \tau} \lvert \prox_V^\tau (x_1) - x_1 \rvert^2\right) \, \mathrm{d} \vrk \\
    &\geq \cF_V [\vrkp] + \frac{1}{2 \tau} W_2^2 (\vrkp, \vrk)\,.
    \end{align*}
    Thus, $\vrkp$ is a minimizer of the scheme, as the proximal operator provides a unique solution to the strictly convex optimization problem.
\end{proof}

\begin{lem}
    Step {\rm (2)} of Algorithm~\ref{alg:dgf} is a~step of length $\tau$ along the  gradient flow of the entropy functional~$\cFE$.
\end{lem}

\begin{proof}
Note that from Lemma~\ref{lemma:dmm3}{\it (v)} it follows that $\cFE$ is geodesically convex and \cite[Theorem 11.1.4]{Ambrosio-Gigli-Savare} shows that
there exists a unique gradient flow $(\mu_t)_{t\geq 0}$ starting at $\vr_{k+\frac{1}{2}}\in\cP_2(\R^d)$ and this curve is the unique solution of the heat equation $\partial_t\mu=\Delta \mu$ (in the sense of distributions), i.e., it is given as a~convolution with $g_\tau$. Then $\mu_t$ is a density of Brownian motion, see e.g. \cite[Theorem 2.1.5]{oksendal2010stochastic}. Consequently, the step between $\vr_{k+\frac{1}{2}}$ and $\vr_{k+1}$ is a Brownian motion of a length $\tau$. 
\end{proof}

Let us provide bounds on moments for a discrete-in-time Markov Chain defined by Algorithm~\ref{alg:dgf}.
\begin{theo}\label{theo:moments} 
   Let $a\geq 0$, $\vr_0\in\cP_2(\R^d)\cap\cP_a(\R^d)$, $V$ be $\lambda_V$-convex function with $\lambda_V>0$  that attains its minimum at $x^*=0$, and for $k=1,\dots,n$ let $X_k$ be as in Algorithm~\ref{alg:dgf}. Then there exists constant $C_{a}=C_{a}(a,\vr_0,x^*,\lambda_V)>0$ such that for all  $0<\tau<1/\lambda_V$ we have  
   \[
   \sup_k \Ex |X_k|^a\leq  C_{a}{d^{\frac{a}{2}}}\lambda_V^{-\frac{a}{2}}\,.
   \]
   Moreover, $C_{a}\ll +\infty$ when $\lambda_V \to 0$.
\end{theo} 

\begin{rem}
\rm
If the minimum of $V$ was not at $x^*=0$, then the result would have been
$$
\sup_k \Ex |X_k-x^*|^a\leq  C_{a}{d^{\frac{a}{2}}}\lambda_V^{-\frac{a}{2}}\,.
$$
\end{rem}

\begin{proof} {We will prove that for some $C_a$, such that $C_a\ll +\infty$ when $\lambda_V \to 0$, it holds
\begin{equation}\label{eq:step-3}
A_{k+1} := \Ex |\Xki|^a \leq C \Big(\tfrac{d}{\lambda_V}\Big)^{\frac{a}{2}}\,.
\end{equation} Steps~1 and~2 provide \eqref{eq:step-3} for $a\in[0,2]$, while in Step 3 we concentrate on $a>2$. Furthermore, in Step~3  we assume, as the induction condition, that~\eqref{eq:step-3} is satisfied for $a-1$ and we infer that the induction condition is true for $a$.}\newline

\noindent {\bf Step 1. } We show that $\Ex |\Xki|^2 \leq \widetilde{C} \tfrac{d}{\lambda_V}$ for $\widetilde{C}$ such that $\widetilde{C} \to 4$ when $\lambda_V \to 0$.
We note that 
\begin{equation*}
\Ex|\Xki|^2 =\Ex|\Xkip + Z_{k+1}|^2 = \Ex | \Xkip |^2 + \Ex |Z_{k+1}|^2
 + 2 \Ex [\Xkip \cdot Z_{k+1}] \, .
\end{equation*}
We take into consideration that $\Xkip$ and $Z_{k+1}$ are independent and we apply Corollary~\ref{cor:prox-contracts} and Lemma~\ref{lemma:ath_mom_bounds} (cf. Remark~\ref{rem:ath_mom_bounds}) to get 
\begin{flalign*}
\Ex|\Xki|^2 &\leq \tfrac{1}{1+\tau \lambda_V}  \Ex|\Xk|^2 + 2 \tau d\leq \left(\tfrac{1}{(1 + \tau \lambda_V)^2}\right)^{k+1} \Ex |X_0|^2 + 2 \tau d \sum_{i=0}^{k} \left( \tfrac{1}{(1 + \tau \lambda_V)^2} \right)^i \, .
\end{flalign*}
Since by assumption $\tau<1/\lambda_V$, we have that $\tfrac{1}{1 + \tau \lambda_V} \leq 1$,  
\begin{equation*}
\sum_{i=0}^k \big( \tfrac{1}{(1+\tau \lambda_V)^2} \big)^i \leq \sum_{i=0}^{+\infty} \big( \tfrac{1}{(1+\lambda_V\tau)^2} \big)^i \leq \tfrac{(1+\lambda_V)^2}{(1+\lambda_V\tau)^2-1} = \tfrac{(1+\lambda_V)^2}{\lambda_V\tau (2+\lambda_V\tau)}\leq \tfrac{4}{3} \tfrac{1}{\lambda_V\tau} 
\end{equation*}
and we get 
\begin{equation*} 
\Ex|\Xki|^2 \leq \Ex|X_0|^2 + \tfrac{8}{3}\tfrac{d}{\lambda_V} = \left( \tfrac{\lambda_V}{d} \Ex|X_0|^2 + \tfrac{8}{3} \right) \tfrac{d}{\lambda_V}\,.
\end{equation*}
Now we denote the term in the brackets as constant $\widetilde{C} := \tfrac{\lambda_V}{d} \Ex|X_0|^2 + \tfrac{8}{3}$.\newline

\noindent {\bf Step 2. } We aim to bound $\Ex|\Xki|^a$  for $a\in[0,2)$. We use the Jensen inequality and  Step~1 to obtain
\begin{equation*}
\Ex |\Xki|^a = (\Ex |\Xki|^2)^{\frac{a}{2}}\leq \left(\widetilde{C} \tfrac{d}{\lambda_V}\right)^{\frac{a}{2}}\,.
\end{equation*}

\medskip
 
\noindent {\bf Step 3. } {Suppose $a>2$.}  Let us assume that \eqref{eq:step-3} holds for $a-1$. By Lemma~\ref{lem:power-est}  with  $X:= \Xkip$ and $Z:=Z_{k+1}$ we get 
\begin{flalign*}
A_{k+1} &:= \Ex |\Xkip + Z_{k+1}|^a \leq \Ex|\Xkip|^a + C \Ex \Big[|\Xkip |^{a-2} |Z_{k+1}|^2 \Big] + C \Ex |Z_{k+1}|^a \\ &=: \mathrm{I} + \mathrm{II} + \mathrm{III} \, .
\end{flalign*}
Note that  by Corollary~\ref{cor:prox-contracts} we can estimate
\begin{equation*}
\mathrm{I} = \Ex|\Xkip|^a \leq \Big(\tfrac{1}{1+ \tau \lambda_V} \Big)^{a}  \Ex | \Xk| ^a \,.
\end{equation*}
As for $\mathrm{II}$, we notice that, by the induction assumption, \eqref{eq:step-3} holds for $a-2$. Since $\Ex |Z_{k+1}|^2 = 2\tau d$ (cf. Lemma \ref{lemma:ath_mom_bounds} and Remark \ref{rem:ath_mom_bounds}), we get that 
\begin{equation*}
\mathrm{II} = \Ex |\Xkip|^{a-2} \Ex |Z_{k+1}|^2 \leq C \Big(\tfrac{d}{\lambda_V} \Big)^{\frac{a-2}{2}} d \tau = C d^{\frac{a}{2}} \lambda^{-\frac{a-2}{2}} \tau\,.
\end{equation*}
We find a bound for $\mathrm{III}$ using Lemma~\ref{lemma:ath_mom_bounds} and recalling the assumption that $\tau< 1/\lambda_V$. Namely, we obtain 
\begin{equation*}
\mathrm{III} \leq C d^{\frac{a}{2}} \tau^{\frac{a}{2}} \leq C d^{\frac{a}{2}} \tau \lambda_V^{-\frac{a-2}{2}} \, .
\end{equation*}
Altogether we have 
\begin{equation*}
A_{k+1} = \mathrm{I} + \mathrm{II} + \mathrm{III} \leq 
\tfrac{1}{(1+\tau \lambda_V)^2} A_k + C d^{\frac{a}{2}} \tau \lambda_V^{-\frac{a-2}{2}} \, .
\end{equation*}
By iteration and noticing again that $\sum_{j=0}^k\Big( \tfrac{1}{(1+\tau \lambda_V)^2} \Big)^{j}\leq \frac{4}{3}\frac{1}{\tau\lambda_V}$, we get 
\begin{equation*}
A_{k+1} \leq \Big( \tfrac{1}{(1+\tau \lambda_V)^2} \Big)^{k+1} A_0 + C d^\frac{a}{2}\tau \lambda_V^{-\frac{a}{2}+1}\sum_{j=0}^k\Big( \tfrac{1}{(1+\tau \lambda_V)^2} \Big)^{j}\leq  
A_0 + C d^\frac{a}{2}\lambda_V^{-\frac{a}{2}}\,.
\end{equation*}
We complete the arguments the same way as in Step 1.
\end{proof}

\begin{coro}\label{coro:moments}
Let $a\geq 0$, $\vr_0\in\cP_2(\R^d)\cap\cP_a(\R^d)$, $V$ be $\lambda_V$-convex function with $\lambda_V>0$  that attains its minimum at $x^*=0$, and for $k=1,\dots,n$ let $\vrk,\vr_{(k-1)+\frac{1}{2}}$ be as in Algorithm~\ref{alg:dgf}. Then there exists constant $C_{a}=C_{a}(a,\vr_0,x^*,\lambda_V)>0$ such that for all  $0<\tau<1/\lambda_V$ we have  
\[
\sup_k \left(\vrk(|\cdot|^a)+\vr_{(k-1)+\frac{1}{2}}(|\cdot|^a)\right)\leq C_{a}d^{\frac{a}{2}}\lambda_V^{-\frac{a}{2}} \,.
\]
Moreover, $C_{a}\ll +\infty$ when $\lambda_V \to 0$.

\end{coro}
\begin{proof}
The bound for $\vrk$ is a direct consequence of Theorem~\ref{theo:moments}. In the case of $\vr_{(k-1)+\frac{1}{2}}$, we note that by Theorem~\ref{theo:moments} and Lemma \ref{lemma:ath_mom_bounds} (with $\widetilde{C}_a$),
\[
\Ex|X_{(k-1)+\frac{1}{2}}|^2=\Ex|X_{k}|^2+\Ex|Z_{k}|^{2}-2\Ex[X_k\cdot Z_k]\leq (2+\widetilde{C}_{a})\lambda_V^{-1}d\,.
\]
For $a\in[0,2)$, we use the Jensen inequality to get
$$
\Ex|X_{(k-1)+\frac{1}{2}}|^a\leq \Big((2+\widetilde{C}_{a})\lambda_V^{-1}d\Big)^{\frac{a}{2}}\,.
$$
Finally, for $a>2$, an application of Theorem~\ref{theo:moments}, Lemma \ref{lemma:ath_mom_bounds}, and $\tau<1/\lambda_V$ then yields
\[
\Ex|X_{(k-1)+\frac{1}{2}}|^a\leq 2^{a-1}\Ex|X_{k}|^a+2^{a-1}\Ex|Z_{k}|^{a}\leq 2^{a-1}(C_{a}+\widetilde{C}_a)d^{\frac{a}{2}}\lambda_V^{-\frac{a}{2}}\,.\qedhere
\]
\end{proof}

Let us now establish the properties of $\cF$ that play a key role in the proof of the convergence of Algorithm~\ref{alg:dgf}. This lemma is inspired by the methods of~\cite{DMM}.

\begin{lem}\label{lemma:dmm3}
Let the function $V: \R^d \to \R$, $d\geq 1$, satisfies assumption \ref{V} with $\lambda_V\geq 0$, $L_{q_V}>0$, and $q_V\geq 1$. Suppose $\cFE$ is given by \eqref{cFE}, $\cFV$ by \eqref{cFV}, and $\cF$ by \eqref{cF} (with $W\equiv0$). Assume further that  $\muinfty$ is a minimizer of $\cF$, $\tau>0$, $\vr_0 \in\cP_2(\R^d)$, and for $k=0,\ldots,n-1$ we have $\vr_{k+\frac{1}{2}} = (\prox^\tau_{V})_{\#} \vr_k $, $ \vr_{k+1} = \vr_{k+\frac{1}{2}} * g_{ \tau}$.  Then: 
\begin{enumerate}[{(i)}]
    \item Along generalized geodesics on $\cP_2 (\R^d)$ we have that $\cFE$ is $0$-convex, while $\cFV$  and $\cF$   are $\lambda_V$-convex. 
 
    \item For any $\nu\in\cP_2(\R^d)\cap\cP_{q_V-1}(\R^d)$, we have 
    \[ 
    2\tau\big(\cFV[\nu*g_{\tau}]-\cFV[\nu]\big)\leq 2^{4q_V-2}L_{q_V}\left(\left(1 + \nu(|\cdot|^{q_V-1})\right)2d\tau^2+C_{q_V+1}d^{\frac{q_V+1}{2}}\tau^{\frac{q_V+3}{2}}\right).
    \]
    \item For any $\nu\in \cP_2(\R^d)$ and $k\in\N$, we have  
    \[2 \tau \big( \cFV [\vrkp] - \cFV [\nu]  \big) \leq W_2^2 (\vrk, \nu) - W_2^2 (\vrkp, \vrk) - (1+\tau \lambda_V )W_2^2 (\vrkp, \nu)\,,\]
    \item For any $ \nu \in \cP_2 (\R^d)$ and $k\in\N$, we have 
    $$2\tau \big( \cFE [\vrki] - \cFE [\nu] \big) \leq W_2^2(\vrkp, \nu) - W_2^2 (\vrki, \nu)\,.$$
    \item For any $ \nu \in \cP_2 (\R^d)$, any $k\in\N$, and under the restriction $0<\tau<1/\lambda_V$, we have
    \begin{flalign*}
        2\tau \big( \cF [\vrki] - \cF [\nu]\big)  &\leq  W_2^2 (\vrk, \nu) - \lambda_V \tau W_2^2(\vrkp, \nu)- W_2^2 (\vrkp, \vrk) - W_2^2(\vrki, \nu)\\ &\quad+ 2^{4q_V-1} L_{q_V}(d+C_{q_V-1}d^{\frac{q_V+1}{2}}\lambda_V^{-\frac{q_V-1}{2}})\tau^2 +
    o(\tau^2) \, .
    \end{flalign*} 
    \item For any $k\in\N$ and under the restriction $0<\tau<1/\lambda_V$, we have 
    \begin{flalign*}W_2^2(\vrki, \muinfty) &\leq   W_2^2 (\vrk, \muinfty) - \lambda_V \tau W_2^2(\vrkp, \muinfty) - W_2^2 (\vrkp, \vrk) \\ &\quad +  2^{4q_V-1} L_{q_V}(d+C_{q_V-1}d^{\frac{q_V+1}{2}}\lambda_V^{-\frac{q_V-1}{2}})\tau^2 +
    o(\tau^2) \, .
    \end{flalign*} 
\end{enumerate} 
\end{lem}

\begin{proof} In order to motivate {\it (i)} we observe that \cite[Proposition~9.3.2]{Ambrosio-Gigli-Savare} implies
\begin{equation}\label{eq:Fconv1}
\cFV[\mu_\theta^{2\rightarrow 3}]  \leq \theta \cFV [\mu^2]+ (1-\theta)\cFV[\mu^3] - \tfrac{\lambda_V}{2} \theta (1-\theta) W_{\boldsymbol{\mu}}^2 (\mu^2, \mu^3)\,,
\end{equation}
while by \cite[Proposition~9.3.9]{Ambrosio-Gigli-Savare} we have
\begin{equation}\label{eq:Fconv2}
\cFE[\mu_\theta^{2\rightarrow 3}]  \leq \theta \cFE [\mu^2]+ (1-\theta)\cFE[\mu^3] \,.
\end{equation}
Summing \eqref{eq:Fconv1} and \eqref{eq:Fconv2} we get the desired $\lambda_V$-convexity of $\cF$.\\
As for {\it (ii)} we observe that by Lemma~\ref{lem:grad-f-q-growth} applied to $\nabla V$ (and then that $\nabla V(x)\cdot y$ is odd in $y$) we have 
    \begin{align*}
    \cFV[\nu*g_{\tau }]&-\cFV[\nu] \\
    &=
    (4\pi \tau)^{-d/2}\int\int \left(V(x+y)-V(x)\right)\exp(-\tfrac{|y|^2}{4\tau})\dy \,\mathrm{d}\nu(x)\\
    &\leq 2^{2q_V-2} L_{q_V} (4\pi \tau)^{-d/2}\int\int (1+|x+y|^{q_V-1}+|{x}|^{q_V-1})|y|^2\exp(-\tfrac{|y|^2}{4\tau})\dy \,\mathrm{d}\nu(x) \\
    &\leq 2^{2q_V-2} L_{q_V} (4\pi \tau)^{-d/2} \times \\ 
    &\quad\times\int\int \left(1+ \max\{2,1+2^{q_V-2}\}|x|^{q_V-1} + \max\{1,2^{q_V-2}\}|y|^{q_V-1}\right)|y|^2\exp(-\tfrac{|y|^2}{4\tau})\dy \,\mathrm{d}\nu(x) \\
    &\leq 2^{2q_V-2} L_{q_V} \left(2\tau d + \max\{2,1+2^{q_V-2}\}\nu(|\cdot|^{q_V-1})2\tau d+\max\{1,2^{q_V-2}\}C_{q_V+1}d^{\frac{q_V+1}{2}}\tau^{\frac{q_V+1}{2}}\right)\\
    &\leq \max\{2\cdot2^{2(q_V-1)},2^{2(q_V-1)}+2^{4(q_V-1)}\} L_{q_V} \left(2\tau d + \nu(|\cdot|^{q_V-1})2\tau d+C_{q_V+1}d^{\frac{q_V+1}{2}}\tau^{\frac{q_V+1}{2}}\right)\\
    &\leq 2^{4q_V-3}L_{q_V} \left(\left(1 + \nu(|\cdot|^{q_V-1})\right)2d\tau+C_{q_V+1}d^{\frac{q_V+1}{2}}\tau^{\frac{q_V+1}{2}}\right).
    \end{align*}
Note that the moments of the Gaussian were calculated using Lemma~\ref{lemma:ath_mom_bounds} (and Remark \ref{rem:ath_mom_bounds}).  
\medskip

Let us justify {\it (iii)} in a similar way as \cite[Lemmas~4 and~29]{DMM}. We aim to bound $V(\prox_V^\tau (x)) - V(y)$ using Lemma \ref{lem:prox-contracts-1}. Indeed, it follows that

\begin{equation*}
2 \tau \Big( V(\prox_V^\tau (x)) - V(y) \Big) \leq \lvert x - y \rvert^2  -  \lvert x-\prox_V^\tau (x) \rvert^2 - (1+\tau \lambda_V) \lvert \proxv (x)-y \rvert^2 \, .
\end{equation*}
Now, let $\boldsymbol{\gamma} \in \Gamma_o (\vr_k, \nu)$ be the optimal coupling between the probability measures $\vr_k$ and $\nu$. Taking the integral $\int \,\mathrm{d} \boldsymbol{\gamma}$ and considering the facts
\begin{equation*}
W_2^2 (\vrkp, \nu) \leq \mathbb{E} [\lvert \proxv(X) - Y \rvert^2] \quad \mbox{and} \quad W_2^2 (\vrk, \vrkp) \leq \mathbb{E} [\lvert X - \proxv(X) \rvert^2]
\end{equation*}
we complete the proof of {\it (iii)}.\\
Note that {\it (iv)} results from \cite[Lemma~5]{DMM}.\\
In order to show {\it (v)} we will employ {\it (ii)} for $\nu=\vrkp$ (and then Corollary~\ref{coro:moments}), {\it (iii)}, and {\it (iv)} after observing that 
\begin{flalign*}
    \cF[\vrki]-\cF[\nu]& = 
    \Big( \cFV[\vrki]-\cFV[\vrkp]\Big)+\Big(\cFV[\vrkp]-\cFV[\nu]\Big)+\Big(\cFE[\vrki]-\cFE[\nu]\Big)\,. 
\end{flalign*}
Finally, we pick $\nu=\muinfty$, which is a minimizer of $\cF$, to get {\it (vi)} directly from {\it (v)}.
\end{proof}

We are in the position to prove Theorem \ref{thm:alg_bound-intro}. Recall that, as explained in Section \ref{sec:prelim}, the process $Y_t$ obtained from the nonlinear SDE \eqref{eq:mckean_vlasov:model} has a density $\mu_t$ which solves the nonlinear PDE \eqref{eq:model}.

\begin{proof}[Proof of Theorem \ref{thm:alg_bound-intro}]
We start with showing that there exists $C=C(q_V,\mu_0)>0$ such that for $k=1,\ldots,n$, we have 

    \begin{equation}        
\label{eq:prop-small-T} 
    W_2^2 (\vr_k, \mu(t)) \leq 
    6\tau \left( \cF [\mu_0] - \cF [\muinfty] + k C  \Big(d+C_{q_V-1}d^{\frac{q_V+1}{2}}\lambda_V^{-\frac{q_V-1}{2}}\Big)\tau +
    o(\tau)\right)\, .
    \end{equation}

    We will employ Theorem 4.1 in \cite{bernton}. It reads
    \begin{equation*}
    W_2^2(\vr_k , \mu(t)) \leq 
     6\tau\big(\cF [\mu_0] - \cF [\muinfty]+ \Delta_{\tau}^{k}\big) \, ,
    \end{equation*}
    whenever $\tau\Delta_{\tau}^k\to0$ as $\tau\to0$, where 
    $$
    \Delta_{\tau}^{k} := \sum_{i = 0}^{k-1} \cFV[\vr_{i+1}] - \cFV[\vr_{i+\frac{1}{2}}]= \sum_{i = 0}^{k-1} \cFV[\vr_{i+\frac{1}{2}}* g_{\tau}]-\cFV[\vr_{i+\frac 12}]\,. 
    $$
    By Lemma~\ref{lemma:dmm3}{\it (ii)} with $\nu=\varrho_{i+\frac12}$ and then Corollary~\ref{coro:moments}, we have
    \begin{equation*}
     \Delta_{\tau}^k \leq k\left(2^{4q_V-1} L_{q_V}\Big(d+C_{q_V-1}d^{\frac{q_V+1}{2}}\lambda_V^{-\frac{q_V-1}{2}}\Big)\tau +
    o(\tau) \right),
    \end{equation*}
    and we immediately obtain \eqref{eq:prop-small-T}.

Let us now show that for $k=2,\ldots,n$, the following inequality holds true
\begin{equation}
\label{eq:prop-large-T} 
W_2^2 (\vrk, \mu(t))  \leq 4 \mathrm{e}^{-k\tau\lambda_V/2} W_2^2(\mu_0,\muinfty) +  C \Big(d\lambda_V^{-1}+C_{q_V-1}d^{\frac{q_V+1}{2}}\lambda_V^{-\frac{q_V-1}{2}-1}\Big)\tau +o(\tau)\,.
\end{equation} 
By the triangle inequality and the $\lambda_V$-contraction for \eqref{eq:model} due to \cite[Theorem 11.2.1]{Ambrosio-Gigli-Savare}, we get 
\begin{equation}\label{proof-large-T:1st-split}
W_2^2 (\vrk, \mu(t)) \leq 2 W_2^2(\vrk, \muinfty)+2W_2^2(\muinfty, \mu(t))\leq  2 W_2^2(\vrk, \muinfty)+2  \mathrm{e}^{-2\lambda_V t} W_2^2(\mu_0,\muinfty)  \,.
\end{equation}
In order to estimate the last term on the right-hand side we will prove that the following inequality holds
\begin{equation}\label{proof-large-T:from-vrk-to-pi}
W_2^2(\vr_k, \muinfty) \leq \mathrm{e}^{-k\tau\lambda_V/2} W_2^2 (\mu_0, \muinfty) +  C \Big(d\lambda_V^{-1}+C_{q_V-1}d^{\frac{q_V+1}{2}}\lambda_V^{-\frac{q_V-1}{2}-1}\Big)\tau+o(\tau)\, .
\end{equation}
Again, by the triangle inequality,
\begin{equation*}
\tfrac{1}{2} W_2^2(\vrk, \muinfty) \leq W_2^2 (\vrk, \vrkp) +  W_2^2(\vrkp, \muinfty)\,,
\end{equation*}
Lemma~\ref{lemma:dmm3}{\it (vi)}, and setting 
\[
\mathcal{K}:= 2^{4q_V-1} L_{q_V}\Big(d+C_{q_V-1}d^{\frac{q_V+1}{2}}\lambda_V^{-\frac{q_V-1}{2}}\Big)\tau^2+o(\tau^2),
\] 
we get 
\begin{flalign*}\notag
W_2^2(\vr_{k}, \muinfty) &\leq  W_2^2 (\vr_{k-1}, \muinfty) - \lambda_V \tau \big( W_2^2 (\vr_{k-1+\frac{1}{2}}, \vr_{k-1})+W_2^2(\vr_{k-1+\frac{1}{2}}, \muinfty)\big)  + \mathcal{K}  
  \\ 
    &\leq W_2^2 (\vr_{k-1}, \muinfty) - \tfrac{\lambda_V \tau}{2}  W_2^2(\vr_{k-1}, \muinfty)   + \mathcal{K}
 = \left(1 - \tfrac{\lambda_V \tau}{2}\right)  W_2^2(\vr_{k-1}, \muinfty) +  
\mathcal{K}\, . 
\end{flalign*}
By iteration of the above estimate

 we obtain 
\begin{equation*}
W_2^2(\vr_k, \muinfty) \leq \left( 1 - \tfrac{\lambda_V \tau}{2} \right)^k W_2^2 (\mu_0, \muinfty) + \mathcal{K} \sum_{i=0}^{k-1} \left( 1 - \tfrac{\lambda_V \tau}{2} \right)^i\leq \mathrm{e}^{-k\tau\lambda_V/2} W_2^2 (\mu_0, \muinfty) + \mathcal{K} \tfrac{2}{\tau \lambda_V} \, ,
\end{equation*}
which implies~\eqref{proof-large-T:from-vrk-to-pi}. Inserting it into \eqref{proof-large-T:1st-split} gives
$$
W_2^2 (\vrk, \mu(t)) \leq \Big(2  \mathrm{e}^{-2\lambda_V t}+2  \mathrm{e}^{-k\tau\lambda_V/2}\Big) W_2^2(\mu_0,\muinfty) + C \Big(d\lambda_V^{-1}+C_{q_V-1}d^{\frac{q_V+1}{2}}\lambda_V^{-\frac{q_V-1}{2}-1}\Big)\tau+o(\tau)  \,.
$$
In order to get \eqref{eq:prop-large-T}, we need the bound $\mathrm{e}^{-2\lambda_V t}\leq \mathrm{e}^{-k\tau\lambda_V/2}$, which we obtain for $k>1$ since $4t\geq k\tau$.

To continue, we choose $\tau\leq \tau_0$ such that the $o(\tau)$-terms in \eqref{eq:prop-small-T} and \eqref{eq:prop-large-T} can be bounded by the second to last terms. Finally, we do elementary steps to obtain the desired estimate.
\end{proof}

Summarizing all the results of this section we can prove our main result concerning Algorithm~\ref{alg:particle}.
\begin{proof}[Proof of Theorem \ref{thm:alg_bound-nonlocal}]
We use the triangle inequality
$$
W_2^2(\vr_k^N,\mu(t))\leq 2 W_2^2(\vr_k^N,\mu^N(t))+2W_2^2(\mu^N(t),\mu(t)).
$$
The last term can be estimated directly by Theorem \ref{theo:chaos}, so we focus on the first term.

We see that in Algorithm \ref{alg:particle} function $\boldsymbol{\Psi}$ plays the same role as $V$ has in Algorithm \ref{alg:dgf}. Moreover, by Lemmas~\ref{lem:Psi1} and~\ref{lem:Psi2}, $\boldsymbol{\Psi}$ is $\lambda_V$-convex and $\nabla\boldsymbol{\Psi}$ is of $q_{\boldsymbol{\Psi}}=\max\{q_V,q_W\}=\widetilde{q}_1$ Lipschitz growth. Hence, we can use Theorem \ref{thm:alg_bound-intro} to estimate $W_2^2(\boldsymbol{\vr}_k,\boldsymbol{\mu}(t))$, and then that $W_2^2(\vr_k^N,\mu^N(t))\leq W_2^2(\boldsymbol{\vr}_k,\boldsymbol{\mu}(t))$ to obtain the result.
\end{proof}

\begin{rem}[Relaxing $\lambda$-convexity]\rm \label{rem:relax-lambda-conv-2} Note that $\boldsymbol{\Psi}$ can be $\lambda_{\boldsymbol{\Psi}}$-convex for $\lambda_{\boldsymbol{\Psi}}>0$ even if $\lambda_W<0$, see Lemma~\ref{lem:lambda-saving}.
Then, due to the Jensen inequality, Theorem~\ref{thm:alg_bound-intro} can be formulated for convergence in $W_1$. By combining such modified result with \cite[Corollary~3]{durmus_propagation} one can obtain the counterpart of Theorem~\ref{thm:alg_bound-nonlocal} for some $\lambda_W<0$, cf. Remark~\ref{rem:relax-lambda-conv}.
\end{rem}

\section{Approximated Proximal Split Discretized Gradient Flow}\label{sec:5}
Analytical computation of $\proxv$ operator is rarely possible. However, effective approximation algorithms for $\proxv$ are available. Our algorithms are stable in the sense that an error coming from such approximation can be controlled without a loss of convergence rate. We will show it for Algorithm~\ref{alg:dgf}, see Theorem~\ref{theo:linear-problem-stability}. Nonetheless, the same holds true for  Algorithm~\ref{alg:particle} by the same arguments. Additionally, we show that  $\proxv$  can be accurately approximated by the gradient descent algorithm, i.e., Algorithm~\ref{alg:gdwb}. The rate of convergence of this algorithm is provided in Corollary~\ref{coro:grad-desc}. Let us consider a modified (perturbed) Algorithm~\ref{alg:dgf}, namely Algorithm~\ref{alg:dgf-pert}.
\begin{algorithm}
\caption{Approximated Proximal Split Discretized Gradient Flow}\label{alg:dgf-pert}
\begin{algorithmic}
\State Sample initial distribution $\widehat X_0 \sim \vr_0=\mu_0$
\For{$k=0, \ldots , n-1$}
\State  $\begin{array}{ll}
(1) \quad \Xpkph := \prox^\tau_V (\Xpkh)+\xi_k\,;\qquad&\text{$\xi_k$ -- perturbation such that $\sup_k|\xi_k|<\ve\,$,} 
\\
(2) \quad \Xpkoh := \Xpkph+ Z_{k+1}\,; & Z_{k+1} \sim g_{ \tau}\,, 
\end{array}$
\EndFor
\end{algorithmic}
\end{algorithm}

The convergence rate for Algorithm~\ref{alg:dgf-pert} are given in the following counterpart of Theorem~\ref{thm:alg_bound-intro}.
\begin{theo} \label{theo:linear-problem-stability}
    Let $\muinfty$ be the unique minimizer of $\cF$ given by~\eqref{cF} (with $W\equiv0$) for $V$ that satisfies \ref{V} with $\lambda_V > 0$. Assume further that $\mu$ is a gradient flow solution to~\eqref{eq:model} with initial datum $\mu_0\in\cP_2(\R^d)\cap \cP_{q_V-1}(\R^d)$. For every $k=1,\ldots,n\,$ let $\Xpkh$ be an output of Algorithm~\ref{alg:dgf-pert}  and $\Xpkh\sim\widehat\vr_k$.    
    Then for $0<\tau < 1/\lambda_V$ and $t\in(\tau(k-1), \tau k]$ the following inequality holds true
\begin{equation}\label{eq:theo-nonexact}
W_2^2 (\widehat{\vr}_k, \mu(t)) \leq 2 W_2^2(\vrk, \mu(t)) + 8\big(\tfrac{\ve}{\lambda_V \tau} \big)^2\,,
\end{equation}
where $W_2^2(\vrk, \vr(t))$ can be estimated as in Theorem \ref{thm:alg_bound-intro}. 

When in addition $\ve = \mathcal{O}(\tau^2)$, then there exists constant $C$ such that
\begin{equation} \label{eq:theo-nonexact2}
W_2(\widehat{\vr}_k, \mu(t))\leq C\tau\,.
\end{equation}
\end{theo}
\begin{proof}
The triangle inequality gives
\begin{equation*}
W_2^2 (\widehat{\vr}_k, \mu(t)) \leq 2 W_2^2(\vrk, \mu(t)) + 2 W_2^2( \vrk,\widehat{\vr}_k) \, ,
\end{equation*} 
so let us concentrate on the second term. The gradient flows of  $\cFE$ and $\cFV$ are respectively $0$- and $\lambda_V$-contractions by \cite[Theorem 11.2.1]{Ambrosio-Gigli-Savare}, so for $k=0,\dots,n-1$, we have
\begin{equation*}
W_2 (\vrki, \widehat{\vr}_{k+1}) \leq W_2 (\vrkp, \widehat{\vr}_{k+\frac{1}{2}}) \,
\end{equation*}
and
\[
W_2(\vrkp,({\proxv})_\# \widehat \vr_k)\leq \tfrac{1}{1+\lambda_V \tau}W_2(\vrk,\widehat\vr_k)\,.
\]
Let $\ve>0$ be a bound for perturbation $\xi_k$ from Algorithm~\ref{alg:dgf-pert}. To complete the proof it suffices to  show that 
\begin{equation*}
W_2 (\vrkp, \widehat{\vr}_{k+\frac{1}{2}}) \leq \tfrac{2\ve}{\lambda_V \tau} \, .
\end{equation*} 
We show it by iterating estimates. Due to Lemma~\ref{lem:wass-pert} we have that
$$
W_2 (({\proxv})_\# \widehat \vr_k,\widehat{\vr}_{k+\frac{1}{2}} )\leq \ve\,.
$$
Consequently, for $k=0,\dots,n-1$, we get
\begin{align*}
W_2 (\vrki, \widehat{\vr}_{k+1}) \leq W_2 (\vr_{k+\frac{1}{2}}, \widehat{\vr}_{k+\frac{1}{2}}) &\leq W_2 (\vrkp,({\proxv})_\# \widehat \vr_k)+ W_2 (({\proxv})_\# \widehat \vr_k,\widehat{\vr}_{k+\frac{1}{2}} ) \\
&\leq \tfrac{1}{1+\lambda_V \tau} W_2(\vrk,\widehat\vr_k) +\ve \, .
\end{align*}
Iterating, keeping in mind that $\vr_0=\widehat{\vr}_0$, then yields
\[
W_2 (\vr_{k+1}, \widehat{\vr}_{k+1})\leq \ve \sum_{i = 0}^{k-1} \big( \tfrac{1}{1+\lambda_V \tau} \big)^i  \leq \ve 
\sum_{i=0}^{+\infty} \big( \tfrac{1}{1+\lambda_V\tau} \big) ^i = \ve\tfrac{1+\lambda_V \tau}{\lambda_V \tau} \leq \tfrac{2\ve}{\lambda_V \tau}\,. 
\]
That proves \eqref{eq:theo-nonexact}. Combining \eqref{eq:theo-nonexact} with Theorem~\ref{thm:alg_bound-nonlocal}, we get the convergence rate  \eqref{eq:theo-nonexact2}. 
\end{proof}

We will explain how the perturbed Algorithm~\ref{alg:dgf-pert} can be realized via the standard \underline{gradient descent algo\-rithm}. Let
\begin{equation}\label{Px}
    P_x(y):=V(y) + \tfrac{1}{2\tau} \lvert y-x \rvert^2\,.
\end{equation}  
  For fixed $k$, $\Xpkh$ being an output of Algorithm~\ref{alg:dgf-pert} such that  and $\Xpkh\sim\widehat\vr_k$, and $\gamma$ being sufficiently small, we will we approximate $ \proxv (\Xpkh)$ by $(\eta^k_{\ell})_{\ell\in\N}$ given by the iterative procedure from Algorithm~\ref{alg:gdwb}.

\begin{algorithm}
\caption{Gradient Descent With Backtracking}\label{alg:gdwb}
\begin{algorithmic}
\State Set $\ell=0$, $\varsigma\in(0,1)$, $\gamma>0$, $\eta_{0}^{k}=\Xpkh\,$. 
\While{$\ell< n-1$}
\State  $\eta^k_{\ell+1}:=\eta^k_{\ell}-\gamma\nabla P_{\Xpkh}(\eta^k_{\ell})$;
\If{$P_x(\eta^k_{\ell+1})\leq V(\eta^k_0)$ \textbf{ and } $P_x(\eta^k_{\ell+1})\leq P_x(\eta^k_{\ell})+\nabla P_x(\eta^k_{\ell})\cdot(\eta^k_{\ell+1}-\eta^k_{\ell})+\tfrac 1{2\gamma}  \vert \eta^k_{\ell+1}-\eta^k_{\ell} \vert^2\,,$\\
}
\State $\ell:=\ell+1$;
\Else
\State $\gamma:=\varsigma\gamma$.
\EndIf
\EndWhile
\end{algorithmic}
\end{algorithm}
  
We make use of the following fact on the rate of convergence of the gradient descent algorithm, which  is well-known in the global case. We show its local version for the sake of completeness.

\begin{lem}\label{lem:for-grad-desc}
Let $x\in\R^d$, $\tau>0$, $V$ satisfy \ref{V} with $\lambda_V > 0$, while $P_x$ be given by~\eqref{Px}. Assume further that $\eta_0=x$ and for $\ell=0,1,\ldots$,
\begin{equation*}
   \eta_{\ell+1}:=\eta_{\ell}-\gamma\nabla P_{x}(\eta_{\ell})\,.
\end{equation*} 
Then there exists a positive constant $\gamma_0=\gamma_0(x,\tau,q_V)$, such that for $\gamma\in (0, \gamma_0]$ and for every $\ell\in\N$ it holds
    \begin{equation*}
        |\eta_\ell-\proxv (x)|^2\leq \left(1-\gamma({\lambda_V+\tfrac{1}{\tau}})\right)^\ell |x-\proxv (x)|^2\,.
    \end{equation*}
\end{lem}

\begin{proof}
Computing $\proxv (x)$ requires minimizing $P_x(y)$ over $\R^d$, but its minimum is attained inside a~compact set
\begin{equation*}
    K_x:=\{y\in\R^d:\ P_x(y)\leq P_x(x)\}\,.
\end{equation*}
First we show that there exists $\gamma_1>0$ such that for all $\gamma\leq \gamma_1$ we have  that if $\eta_\ell\in K_x$ then $\eta_{\ell+1}\in K_x$.
By Lemma~\ref{lem:grad-f-q-growth} applied to $V$ and the equality
$
|z-x|^2=|y-x|^2+2(y-x)\cdot(z-y)+|z-y|^2,
$
we have
\begin{flalign*}
P_x(\eta_{\ell+1})=P_x\big(\eta_\ell-\gamma\nabla P_x(\eta_\ell)\big)
&\leq P_x(\eta_\ell)-\gamma|\nabla P_x(\eta_\ell)|^2\\
&\quad+\left[2^{2q_V-2}L_{q_V}\left(1+|\eta_\ell|^{q_V-1}+|\eta_\ell-\gamma\nabla P_x(\eta_\ell)|^{q_V-1}\right)+\tfrac{1}{2\tau}\right]\gamma^2|\nabla P_x(\eta_\ell)|^2\,.
\end{flalign*}
Since $\nabla P_x$ is bounded on bounded sets, the term in the bracket is bounded for $\eta_\ell\in K_x$. Consequently, we can choose $\gamma_1$ small enough such that for $\gamma\leq\gamma_1$ we have $\eta_{\ell+1}\in K_x$. We fix $\gamma_0$ such that 
\[
\tfrac{1}{\gamma_0}=2\left(2^{2q_V-2}L_{q_V}\sup_{z,y\in K_x}\left(1+|z|^{q_V-1}+|y|^{q_V-1}\right)+\tfrac{1}{2\tau}\right)\,
\]
and note that $0<2\gamma_0\leq\gamma_1$. By the proof of Lemma~\ref{lem:grad-f-q-growth} we deduce that for every $z,y\in K_x$ and $\gamma\leq\gamma_0$
we have
\begin{equation*}
    P_x(z)\leq P_x(y)+\nabla P_x(y)\cdot(z-y)+\tfrac 1{2\gamma}  \vert z-y \vert^2 \,.
\end{equation*}
Fix   $y\in K_x$. We show now that \begin{equation*}
    \gamma|\nabla P_x(y)|^2\leq 2(P_x(y)-P_x(\proxv(x)))\,.
\end{equation*}
Indeed, it results from the fact that since $\proxv(x)$ minimizes $P_x$, we have
\[P_x(\proxv(x))\leq P_x(y-\gamma \nabla P_x(y))\leq P_x(y)-\gamma|\nabla P_x(y)|^2+\tfrac{\gamma}{2}|\nabla P_x(y)|^2=P_x(y)-\tfrac{\gamma}{2}|\nabla P_x(y)|^2\,.\]
Notice that by its very definition $P_x$ is $\lambda$-convex for $\lambda=\lambda_V+\frac{1}{\tau}$, which means that
\[-2\nabla P_x(y)\cdot(y-\proxv(x))\leq -2(P_x(y)-P_x(\proxv(x)))-({\lambda_V+\tfrac{1}{\tau}})|y-\proxv(x)|^2\,.\]
Summing up, we obtain
\begin{flalign*}
    |\eta_{\ell+1}-\proxv(x)|^2&=
    |\eta_{\ell}-\proxv(x)-\gamma\nabla P_x(\eta_\ell)|^2\\
    &=
    |\eta_{\ell}-\proxv(x)|^2-2\gamma\nabla P_x(\eta_\ell)\cdot (\eta_\ell-\proxv(x))+\gamma^2|\nabla P_x(\eta_\ell)|^2\\
    &\leq 
    |\eta_{\ell}-\proxv(x)|^2- 2\gamma(P_x(y)-P_x(\proxv(x)))\\
    &\qquad-\gamma({\lambda_V+\tfrac{1}{\tau}})|y-\proxv(x)|^2+2\gamma(P_x(y)-P_x(\proxv(x)))\\
    &= \Big[1- \gamma({\lambda_V+\tfrac{1}{\tau}})\Big]
    |\eta_{\ell}-\proxv(x)|^2\,.
\end{flalign*}
We complete the proof by the iteration.
\end{proof}

We are in a position to establish the rate of convergence of Algorithm~\ref{alg:gdwb}.

\begin{coro} \label{coro:grad-desc} 
Suppose $\tau>0$ and $V$ satisfies \ref{V} with $\lambda_V > 0$. Let $\Xpkh$ be an  output of Algorithm~\ref{alg:dgf-pert} and $\eta^k_\ell$ and $\gamma$ be  output of Algorithm~\ref{alg:gdwb}. Then for 
     every $\ell\in\N$ it holds
    \begin{equation*}
        |\eta^k_\ell-\proxv (\Xpkh)|^2\leq \left(1-\gamma({\lambda_V+\tfrac{1}{\tau}})\right)^\ell |\Xpkh-\proxv (\Xpkh)|^2\,.
    \end{equation*}
\end{coro}

\begin{proof}
    Due to Lemma~\ref{lem:for-grad-desc} there exists $\gamma_0$, such that for $\gamma<\gamma_0$  the condition \[P_x(\eta^k_{\ell+1})\leq V(\eta^k_0)\quad\textbf{ and }\quad P_x(\eta^k_{\ell+1})\leq P_x(\eta^k_{\ell})+\nabla P_x(\eta^k_{\ell})\cdot(\eta^k_{\ell+1}-\eta^k_{\ell})+\tfrac 1{2\gamma}  \vert \eta^k_{\ell+1}-\eta^k_{\ell} \vert^2\]
    holds true. In turn, Algorithm~\ref{alg:gdwb} will end with $\gamma\in (0,\gamma_0)$ and the rate of convergence results from Lemma~\ref{lem:for-grad-desc}. 
\end{proof}

\begin{rem}\rm The rate of convergence from Corollary~\ref{coro:grad-desc} can be estimated as follows\[1-\gamma(\lambda_V+\tfrac{1}{\tau})\leq 1-\varsigma\tfrac{\tau\lambda_V+1}{\tau c+1}\,\]
for some constant $c=c(V,x)>0$.
The estimate follows from the proofs of Lemma~\ref{lem:for-grad-desc} and Corollary~\ref{coro:grad-desc} we can infer that
    $\gamma_0=\frac{\tau}{c\tau+1}$ for some $c>0$. Furthermore, $\gamma$ obtained in Algorithm~\ref{alg:gdwb} satisfies $\gamma\geq \varsigma\gamma_0$.\\
     We notice that for all sufficiently small $\tau$ the convergence is arbitrarily fast as \[1-\varsigma\tfrac{\tau\lambda_V+1}{\tau c+1}\xrightarrow[\tau\to 0]{}1-\varsigma\,.\]
\end{rem}

\section{Comments on the possible extensions and open problems}\label{sec:6}
We presented our main results in an illustrative case in order to clearly demonstrate the reasoning involved.   Nonetheless, our approach is tailored for a more general setting. Let us comment on the possible scope of problems that our method allows to tackle. 
\begin{enumerate}[{\it (i)}]

\item If $V$  or $W$ is $1$-growth Lipschitz ($q_V=1$ or $q_W=1$), then by the definition~\eqref{f-q-growth-Lip} this function is Lipschitz. Then wherever in our proofs this property is used, we can skip some of our arguments by making use of the classical convex optimization tools. Consequently, we get the same order of convergence (with possibly improved constants).

\item As observed in the comments right below the formulation of assumptions~\ref{V}--\ref{mu0}, there is no difficulty in allowing for non  differentiable $V$ and $W$ in our technique. Nonetheless, then the sub-gradient notation overburdens the reasoning.

\item Problems involving repulsive interaction potential $W$ are of particular interest, embracing in particular the celebrated case of the Keller--Segel equation, cf. \cite{Hors,DiMarino-Santambrogio,Pert}. Our framework can allow for $\lambda_W$-convex $W$ for some $\lambda_W<0$, which we comment in Remarks~\ref{rem:relax-lambda-conv} and~\ref{rem:relax-lambda-conv-2}. Full analysis in this direction is expected to be challenging, but it should be considered as highly appealing. Numerical simulations are provided in Models B and C in Section \ref{ssec:ex:1d}, and Model G in Section \ref{ssec:ex:2d} below.
\end{enumerate}

\section{Simulations}\label{sec:7}

To validate our analytical results under various conditions, we conduct a series of illustrative simulations. First, we present the influence of various interaction potentials on one-dimensional examples. They are accompanied by a comparison between empirical tests of convergence rates and the theoretical ones obtained in this paper. Then we show two-dimensional examples, where the potentials $V$ and $W$ are not smooth and the confinement potential $V$ is not symmetric. All the simulations have the expected behavior according to our results and the general theory. 

In the examples, we compare local and nonlocal models such that $\lambda_W \geq 0$ (i.e., when our assumptions are satisfied). Additionally, to extend the scenarios, we take into account the case $\lambda_W < 0$, according to Remarks \ref{rem:relax-lambda-conv}~and~\ref{rem:relax-lambda-conv-2}.

Let us also mention that in the case $W\equiv0$, we did some simulations for real-world applications in \cite{BeChEnMi25}. Here we therefore focus on examples appearing in the mathematical literature.

\subsection{One-dimensional models}\label{ssec:ex:1d}
Let $d=1$, and consider the following potentials
\begin{align*}
V_1 (x) = \tfrac{1}{2} |x|^2 &\, , \quad
V_2 (x) = \tfrac{1}{2} |x|\, \boldsymbol{1}_{\{|x|\leq 1\}}+\tfrac{1}{2} |x|^3\, \boldsymbol{1}_{\{|x|> 1\}} \, , \quad\\
W_1(x) = -\tfrac{1}{8} |x|^2 \, , \quad 
&W_2 (x) = \tfrac{1}{3} |x|^3 - \tfrac{1}{8} |x|^2, \quad \, 
W_3 (x) = \tfrac{1}{3} |x|^3, \quad W_4(x)=\tfrac{1}{2}|x|^2\,.
\end{align*} 
The standard power law interaction potentials $W$,
\begin{equation*}
W_k (x) = 
\begin{cases}
\tfrac{\lvert x \rvert^k}{k}, & k \ne 0,\\
\log \lvert x \rvert, &k = 0,
\end{cases}
\end{equation*}
are studied e.g. in \cite[Equation (9)]{Carrillo_review_paper}, while general power-law interaction potentials $W$, 
\begin{equation*}
W(x) = \tfrac{|x|^a}{a} - \tfrac{|x|^b}{b}  \quad -d < b< a \, ,
\end{equation*}
can be found in e.g. \cite[Section 4.3.1]{Carrillo-primal} and \cite[Equation (1.3)]{katy_craig_blob}. We also refer to \cite[Section 3.11]{gomez}. These cases thus form the basis of our numerical simulations:
\begin{enumerate}[A.]
\item  Local model ($W = 0$) with $V_1$.
\item  Repulsive model with $V_1$ and $W_1$.
\item  Attractive-repulsive model with $V_1$ and $W_2$.
\item  Attractive model with $V_1$ and $W_3$.
\item Local model ($W = 0$) with $V_2$.
\item Attractive model with $V_1$ and $W_4$.
\end{enumerate}
%
The confinement potential $V_1$ is 1-convex and its gradient satisfies the $1$-growth condition. The interaction potentials $W_1$ and $W_2$ are $-\frac{1}{4}$-convex and their gradients are $1$- and $2$-growth Lipschitz, respectively. The potential $W_3$ is $0$-convex and its gradient satisfies the $2$-growth Lipschitz condition. It is worth mentioning that the function $\boldsymbol{\Psi}$ (collecting the joint impact of $V$ and $W$ as in~\eqref{def-PsiN}) is $\frac{1}{2}$-convex for Models~B and~C and $1$-convex for Model~D, see Lemma~\ref{lem:lambda-saving}. Thus, in any case, the particle system \eqref{eq:particle_system:model} has the~unique stationary measure $\overline{\boldsymbol{\mu}}_{\infty}$.

We consider the initial condition as a~Gaussian mixture such that 
\begin{equation} \label{eq:initial_1d}
\vr_0(x) = 0.2 * g_{0.5} (x-2) + 0.4* g_{0.5}(x+4) + 0.4* g_{1.125} (x-4) \, .
\end{equation}
We provide the simulation with time step $\tau = 10^{-3}$ and $1\, 000$ particles.
The exact analytical proximal step was used to simulate Model~A (Algorithm~\ref{alg:dgf}) and Model~B (Algorithm~\ref{alg:particle}). To solve Models~C and D, an approximate proximal step was incorporated to overcome the optimization problem (Algorithm~\ref{alg:dgf-pert}). In particular, the Newton conjugate optimization algorithm was used to reduce the number of iterations.

\begin{figure}[htbp]
\pdfimageresolution=100
\captionsetup[subfigure]{labelformat=empty,justification=centering} 

\centering
\begin{subfigure}{0.24\textwidth}
    \captionsetup{justification=centering} \includegraphics[width=\textwidth]{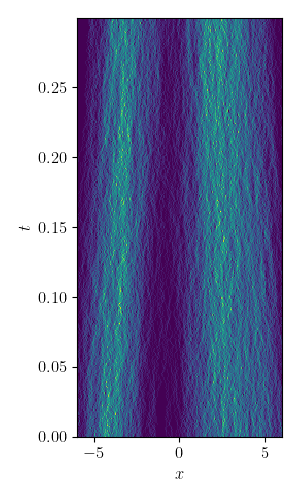}
    \caption{$\qquad$ A}
\end{subfigure}
\begin{subfigure}{0.24\textwidth}
    \includegraphics[width=\textwidth]{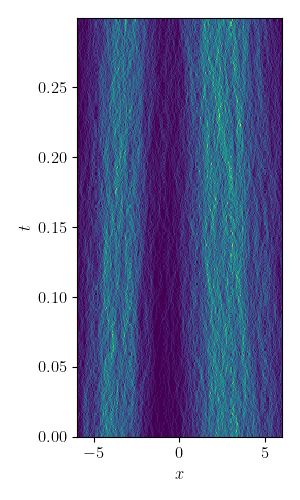}
    \caption{$\qquad$ B}
\end{subfigure}
\begin{subfigure}{0.24\textwidth}
    \includegraphics[width=\textwidth]{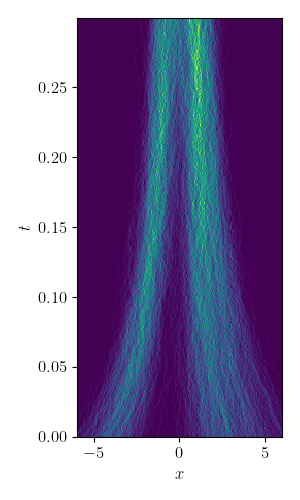}
    \caption{$\qquad$C}
\end{subfigure}
\begin{subfigure}{0.24\textwidth}
    \includegraphics[width=\textwidth]{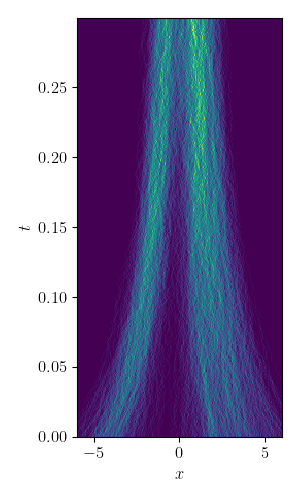}
    \caption{$\qquad$D}
\end{subfigure}

\caption{Evolution of density in one dimensional example with mixture of Gaussian as initial distribution, cf. \eqref{eq:initial_1d}. We show $10\,000$ trajectories for $\tau=10^{-3}$. We run Algorithm~\ref{alg:dgf} for Model 1, Algorithm~\ref{alg:particle}  for Model 2, and Algorithm~\ref{alg:particle} with a perturbed proximal step (as in Algorithm~\ref{alg:dgf-pert}) for Models 3 and 4.}
\label{fig:trajectories_1d}
\end{figure}

The results of the simulations illustrating the one-dimensional problems are presented in Figure~\ref{fig:trajectories_1d}. Let us point out that the first two models,~A and~B, are significantly different than the second two models,~C and~D. This is caused by the presence of an attractive force and the fact that the attractive force is of a~higher order than the force given by the confinement potential $V_1$. \newline

\begin{figure}[htbp]
\pdfimageresolution=100
\centering 
    \includegraphics[width=0.50\textwidth]{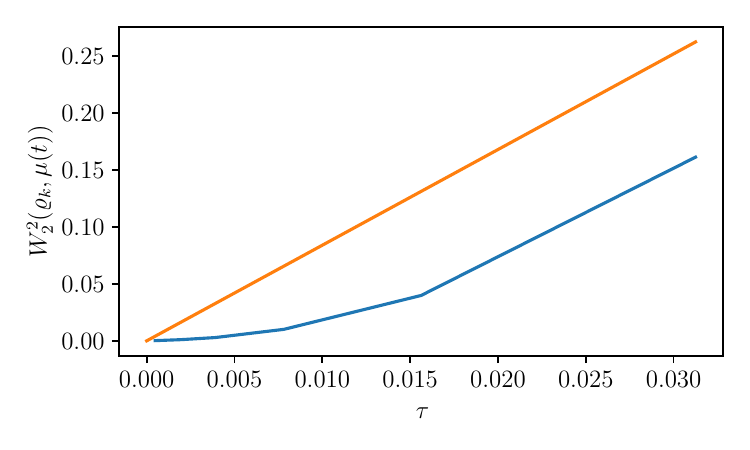}
\caption{Comparison of theoretical convergence rate due (orange) and empirical (blue) for Model E ($W\equiv 0$ and $V_2$) with respect to time step $\tau$, according to Theorem~\ref{thm:alg_bound-intro}, given as average from $15$ replications, for $N=100\,000$. }
\label{fig:convergence_rate1}
\end{figure}

\textit{Convergence rate with respect to time step $\tau$}. We consider the Model E.
We compute an approximate solution with $N=100\,000$ particles and time step $5 * \tau=10^{-8}$ as reference solution. It is worth mentioning that an exact proximal step was used and the comparison is done at the time step $t=0.125$. Figure~\ref{fig:convergence_rate1} presents the empirical solution given as the average of $15$ simulation runs (blue) compared to the theoretical convergence rate given by Theorem~\ref{thm:alg_bound-intro} (orange). Since the squared Wasserstein distance is plotted, we should get $W_2^2 (\vr_k, \mu(t)) \leq \mathcal{O}(\tau)$. We can see that the observed convergence rate is better for shorter time steps $\tau$. 

\textit{Convergence rate with respect to number of particles $N$.}
To validate the convergence rates provided by Theorem~\ref{thm:alg_bound-nonlocal}, we provide two convergence rate tests for both the time step $\tau$ and the number of particles $N$.
We show convergence rates for the attractive Model~F with a chosen time step $\tau$, cf. Figure~\ref{fig:convergence_rate2}.
We compute an approximate solution with time step $\tau=2.25 \times 10^{-7}$ and $2\,000$ particles, and consider it as the reference solution. This solution is then compared with the solutions up to $1\,000$ particles with time step $\tau = 10^{-6}$. An approximate proximal step (Algorithm~\ref{alg:dgf-pert}) was used.
Both cases satisfy the conditions mentioned below Theorem~\ref{thm:alg_bound-nonlocal} and then we should see the $W_2(\rhok^N, \mu(t)) \leq \mathcal{O} (N^{-\frac{1}{2}})$ convergence rate. 
We run 15 replications for the test scenarios and 2 replications for the reference simulation and plot the average of $W_2$-distances.\newline

\begin{figure}[htbp]
\centering 
\pdfimageresolution=100
    \includegraphics[width=0.5\textwidth]{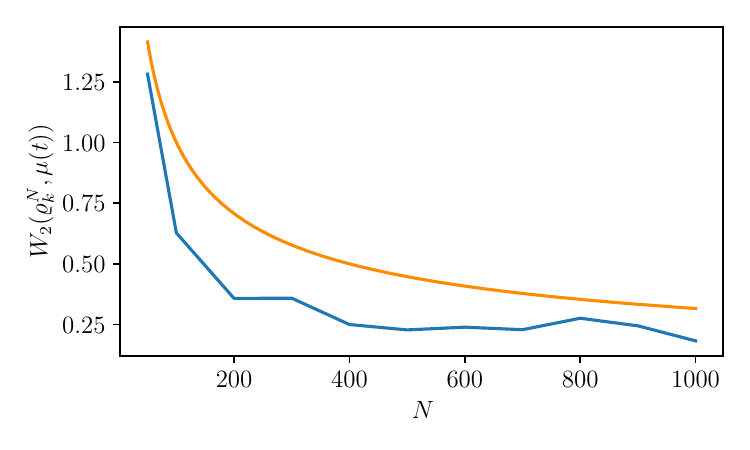}
\caption{Comparison of empirical  convergence rate (blue) and theoretical one (orange) for Model F (attractive) with respect to the number of particles $N$, according to Theorem~\ref{thm:alg_bound-nonlocal}, given as average from $15$ replications, with~$\tau=10^{-6}$.}
\label{fig:convergence_rate2}
\end{figure}

\subsection{Two-dimensional models}\label{ssec:ex:2d}
Let us now perform numerical simulations where we chose potentials within our generalized framework. We consider the confinement potential
\begin{equation*}
V_3(x) = \begin{cases}
(x_1+\tfrac{1}{2})^{4 + \arctan x_1} + x_2^2 - \tfrac{1}{16}, & x_1 \geq 0 \, , \\[2pt]
 \tfrac{1}{4} x_1^2 + x_2^2, & x_1 <0 \, ,
 \end{cases}
\end{equation*}
and the interaction potentials
\begin{equation*}
W_5 (x) = \begin{cases}
    
-\tfrac{1}{8} |x|^2 + 1, & |x| \leq 1 \, ,\\[2pt]
 -|x| + 1, & |x| > 1 \, , 
 \end{cases}
 \qquad
W_6 (x) = 
\begin{cases}
\tfrac{1}{2} |x|,& |x| \leq 1 \, , \\
 \tfrac{1}{2} |x|^3,&  |x| > 1 \, . 
 \end{cases}
\end{equation*}
forming the following models:
\begin{enumerate}

\item[G.] Repulsive model with $V_3$ and $W_5$.
\item[H.] Attractive model with $V_3$ and $W_6$. 
\end{enumerate}
It is worth mentioning that the function $V_3$ is $1$-convex and has a gradient satisfying the  $(3+\frac{\pi}{2})$-growth Lipschitz condition. The function $W_5$ is $(-\frac{1}{4})$-convex and has a global Lipschitz gradient ($q=1$). According to Lemma~\ref{lem:lambda-saving} the function $\boldsymbol{\Psi}$ corresponding to that setting is $\frac{1}{2}$-convex. The function $W_6$ is $1$-convex and its gradient satisfies $2$-growth Lipschitz condition. 

We take the initial condition in the form of a~Gaussian mixture such that
\begin{equation*}
\vr_0(x) = 0.2 \, g_{\frac{1}{2}\Sigma_1} (x - m_1) + 0.4 \, g_{\frac{1}{2}\Sigma_2} (x - m_2) + 0.4 \, g_{\frac{1}{2}\Sigma_3} (x - m_3);
\end{equation*}
\begin{equation*}
m_1 = (4,2),\ m_2 = (-2,-4),\ m_3 = (-2,3) , \ 
\Sigma_1 = \big( \begin{smallmatrix} 1 &0.2 \\ 0.2 & 1.3 \end{smallmatrix} \big) \, , \
\Sigma_2 = \big( \begin{smallmatrix} 1 & -0.2 \\ -0.2 & 1.3  \end{smallmatrix} \big) \, , \
\Sigma_3 = \big( \begin{smallmatrix} 2 & 0.2 \\ 0.2 & 2 \end{smallmatrix} \big) \, .
\end{equation*}
There is an abuse of notation with comparison to density $g_a$ in the Preliminaries section. In this example, we consider the multivariate normal density (measure) and replace the $2a \id$ term with an $2 \Sigma$ term, which is given as a lower index parameter. In this case we run Algorithm~\ref{alg:dgf-pert} with Newton conjugate gradient method.
We provide a simulation with $N=1\, 000$ particles and time step $\tau = 10^{-3}$. We can see the results for both models in Figure~\ref{fig:2d}.

\begin{figure}[htbp] 
\pdfimageresolution=100
\captionsetup[subfigure]{labelformat=empty}
\centering
\begin{subfigure}{0.03\textwidth}
    G \vspace{15mm}
\end{subfigure}
\begin{subfigure}{0.23\textwidth}
    \includegraphics[width=\textwidth]{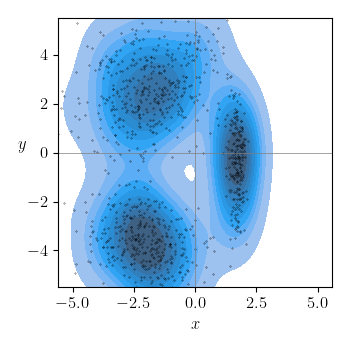}
\end{subfigure}
\begin{subfigure}{0.23\textwidth}
    \includegraphics[width=\textwidth]{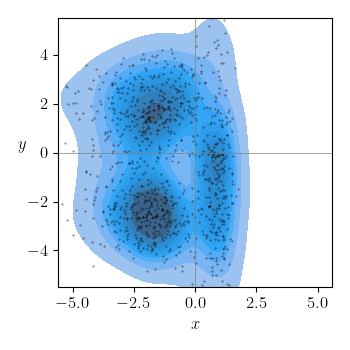}
\end{subfigure}
\begin{subfigure}{0.23\textwidth}
    \includegraphics[width=\textwidth]{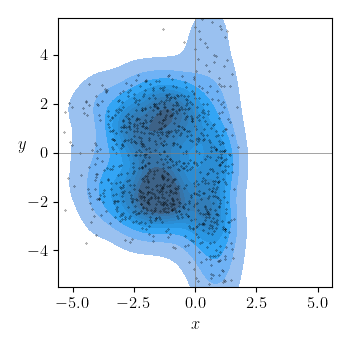}
\end{subfigure}
\begin{subfigure}{0.23\textwidth}
    \includegraphics[width=\textwidth]{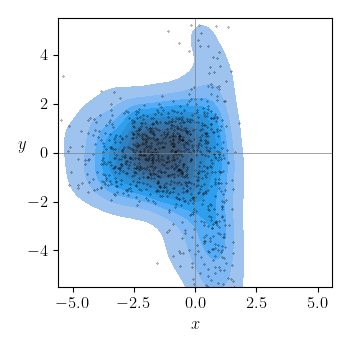}
\end{subfigure}
\begin{subfigure}{0.03\textwidth}
    H \vspace{15mm}
\end{subfigure}
\begin{subfigure}{0.23\textwidth}
    \includegraphics[width=\textwidth]{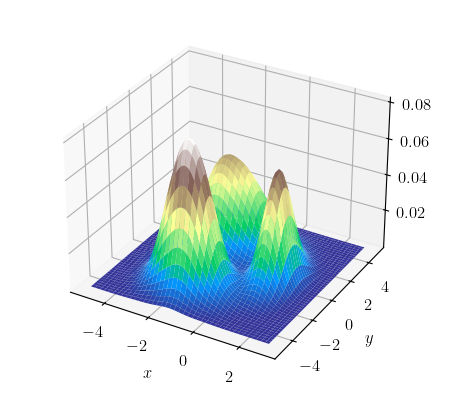}
        \caption{$t = 0.05$}
\end{subfigure}
\begin{subfigure}{0.23\textwidth}
    \includegraphics[width=\textwidth]{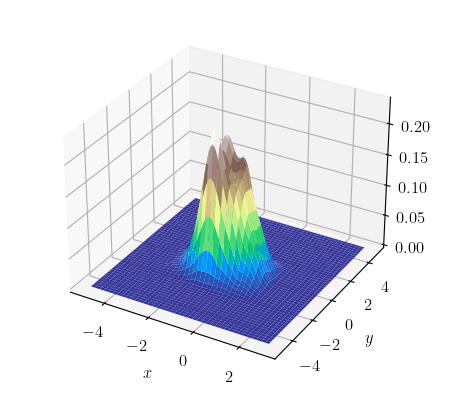}
        \caption{$t = 0.25$}
\end{subfigure}
\begin{subfigure}{0.23\textwidth}
    \includegraphics[width=\textwidth]{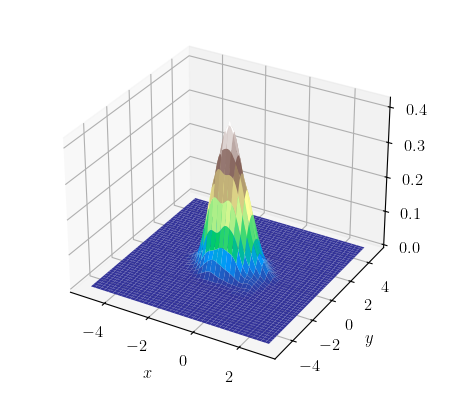}
        \caption{$t = 0.45$}
\end{subfigure}
\begin{subfigure}{0.23\textwidth}
    \includegraphics[width=\textwidth]{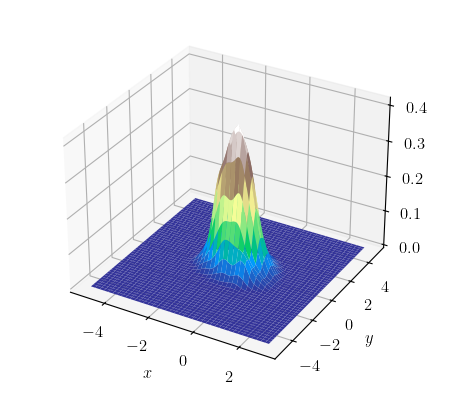}
        \caption{$t = 0.95$}
\end{subfigure}
\caption{Numerical solution given by Algorithm~\ref{alg:particle} with a perturbed proximal step as in Algorithm~\ref{alg:dgf-pert} for Model G (repulsive) and Model H (attractive). Simulation is done with $N=1\, 000$ particles and time step $\tau=10^{-3}$.}
\label{fig:2d}
\end{figure}

\section*{Appendix}

\renewcommand\thesection{\Alph{section}}
\setcounter{section}{1}
\setcounter{coro}{0}

A function $f:\R^d \rightarrow  \R$ is called $\lambda$-{\it convex} for $\lambda\in\R$ on $\R^d$ if for $\theta\in[0,1]$ it satisfies
\begin{equation}
f(\theta x + (1-\theta)y) \leq \theta f(x) + (1-\theta)f(y) - \tfrac{\lambda}{2} \theta (1-\theta) \lvert x - y \rvert ^2 \quad \forall x,y \in \R^d \, .\label{lambda-conv}
\end{equation}

\begin{rem}\rm 
Let $f$ be $\lambda$-convex. Since $f-\tfrac{\lambda}{2}|\cdot|^2$ is $0$-convex, we immediately get that
$$
\big(\nabla f(x)-\nabla f(y)\big)\cdot (x-y)\geq \lambda|x-y|^2\,.
$$
\end{rem}

A function $f:\R^d \rightarrow  \R^d$ is called $q$-{\it growth Lipschitz} on $\R^d$  for $q=1$ if it is Lipschitz, while for $q> 1$ if there exists $L_q>0$ such that 
\begin{equation}\label{f-q-growth-Lip}
|f(x)-f(y)|\leq L_q \min\{|x-y|,1\}(1+|x|^{q-1}+|y|^{q-1}) \quad \forall x,y \in \R^d \, .
\end{equation} 
Note that for $q>1$ the above definition do not require that $f$ is globally Lipschitz.

\begin{lem}\label{lem:grad-f-q-growth} If $\nabla f$ is $q$-growth Lipschitz for $q\geq 1$ with $L_q>0$, then it holds
\[
f(y)\leq f(x)+\nabla f(x)\cdot(y-x)+2^{2q-2} L_q \left(1+|x|^{q-1}+|y|^{q-1}
\right)\vert x-y \vert^2 \quad \forall x,y \in \R^d \, .
\]
\end{lem} 
\begin{proof}
By the Cauchy--Schwartz inequality and the assumption we get that\begin{align*}
    f(y)-f(x) -  \nabla f(x)\cdot(y-x)&=
\int_0^1  \nabla f(x+s(y-x))\cdot(y-x) \ds - \nabla f(x)\cdot(y-x)\\
&=\int_0^1 \left( \nabla f(x+s(y-x))-\nabla f(x)\right)\cdot(y-x) \ds\\
&\leq |x-y|\int_0^1\Big|\nabla f(x+s(y-x))-\nabla f(x)\Big| \ds\\
    &\leq L_q |x-y|^2 \int_0^1 s \big(1+|x+s(y-x)|^{q-1}+|x|^{q-1}\big) \ds \\
     &\leq L_q |x-y|^2 \big(\tfrac{1}{2}+ \tfrac{2^{q-2}+1}{2} |x|^{q-1} +\tfrac{2^{q-2}}{q+1} |y-x|^{q-1} \big)\\
     &\leq L_q \big(  \tfrac{2^{q-2}+1}{2} + \tfrac{2^{2q-4}}{q+1} \big) |x-y|^2 \big( 1+|x|^{q-1} + |y|^{q-1}\big) \,.
\end{align*}
Taking into account that $\tfrac{2^{q-2}+1}{2} + \tfrac{2^{2q-4}}{q+1}  \leq 2^{2q-2}$, we complete the proof.
\end{proof}

\begin{lem}\label{lem:Psi1}
    Let $V: \R^d \to \R$ be $\lambda_V$-convex for some $\lambda_V>0$ and  $W: \R^d \to \R$ be convex and even. Then $\boldsymbol{\Psi}$ given by~\eqref{def-PsiN} is $\lambda_V$-convex.
\end{lem}
\begin{proof}
    We notice that $\nabla \boldsymbol{\Psi}(\boldsymbol{x})\cdot \boldsymbol{x}=\sum_{i=1}^N \nabla V(x^i)\cdot x^i+\frac{1}{N}\sum_{i=1}^N\sum_{j\neq i} \nabla W(x^i-x^j)\cdot x^i$, so by  $\lambda_V$-convexity of $V$, convexity of $W$, and symmetry of $W$, we have 
    \begin{align*}
        \nabla \boldsymbol{\Psi}(\boldsymbol{x})\cdot \boldsymbol{x}&\geq \sum_{i=1}^N \lambda_V|x^i|^2+\frac{1}{N}\sum_{i=1}^N\sum_{j> i} \nabla W(x^i-x^j)\cdot (x^i-x^j)\geq  \lambda_V|\boldsymbol{x}|^2\,.\qedhere
    \end{align*}
\end{proof}

\begin{lem}\label{lem:Psi2}
    Let $V: \R^d \to \R$  and  $W: \R^d \to \R$ have $q_V$- and $q_W$-growth Lipschitz gradients, respectively, then there exist $L_{\boldsymbol{\Psi}}$, such that $\boldsymbol{\Psi}$ given by~\eqref{def-PsiN} satisfies
    \[
    |\nabla \boldsymbol{\Psi}(\boldsymbol{x})-\nabla \boldsymbol{\Psi}(\boldsymbol{y})| \leq L_{\boldsymbol{\Psi}}\min\{|\boldsymbol{x}-\boldsymbol{y}|,1\} (1+|\boldsymbol{x}|^{q_{\boldsymbol{\Psi}}-1}+|\boldsymbol{y}|^{q_{\boldsymbol{\Psi}}-1})\,,
    \]
    where $q_{\boldsymbol{\Psi}}=\max\{q_V,q_W\}$.
\end{lem}

\begin{proof}We have
    \begin{align*}
        |\partial_{x^i} \boldsymbol{\Psi}(\boldsymbol{x})-\partial_{x^i} \boldsymbol{\Psi}(\boldsymbol{y})|&\leq |\nabla V(x^i)-\nabla V(y^i)| +  \frac{1}{N}\sum_{\substack{j=1}}^{N} | \nabla W (x^i-x^j)-\nabla W(y^i-y^j)|\\
        &\leq L_{q_V} \min\{|x^i-y^i|,1\}(1+|x^i|^{q_V-1}+|y^i|^{q_V-1})\\
        &\ \ +L_{q_W} \min\{|(x^i-x^j)-(y^i-y^j)|,1\}(1+|x^i-x^j|^{q_W-1}+|y^i-y^j|^{q_W-1})\\
        &\leq C^i_{\boldsymbol{\Psi}} \min\{|x-y|,1\}(1+|\boldsymbol{x}|^{q_{\boldsymbol{\Psi}}-1}+|\boldsymbol{y}|^{q_{\boldsymbol{\Psi}}-1})\,. \qedhere
    \end{align*}
\end{proof}

\begin{lem} \label{lemma:ath_mom_bounds}
Let $Z \sim g_{\tau}$. Then for every $a\geq 0$, there exists $C_a=C(a)>0$, such that 
\begin{equation*}
\Ex |Z|^a \leq C_a \max\{d,d^\frac{a}{2}\} \tau^{\frac{a}{2}}  \,. 
\end{equation*}
\end{lem}

\begin{proof}
Let $Z=(Z^1,\ldots, Z^d)$ and observe
$$
\Ex |Z|^a =\Ex\left|\sum_{i=1}^d |Z^i|^2\right|^\frac{a}{2}\leq \max\{d^{\frac{a}{2}-1},1\}\sum_{i=1}^d \Ex|Z^i|^a=\max\{d^{\frac{a}{2}},d\} \Ex|Z^1|^a.
$$
 We note that $\tfrac{1}{2\tau}(Z^1)^{2} \sim \chi^2(1) = \mathsf{Gamma}(\tfrac{1}{2}, \tfrac{1}{2})$,  $(Z^1)^2\sim \mathsf{Gamma}(\tfrac{1}{2}, \tfrac{1}{4\tau})$, and
\begin{align*}
\Ex |Z^1|^{a} &  = \int_{0}^{+\infty} y^{\frac{a}{2}} \frac{1}{(4 \tau)^\frac{1}{2} \Gamma (\frac{1}{2})} y ^{\frac{1}{2} -1 } \mathrm{e}^{-\frac{1}{2\tau} y} \d y = 
\frac{(4\tau)^\frac{1+a}{2} \Gamma \left(\frac{1+a}{2}\right)}{(4\tau)^\frac{1}{2} \Gamma (\frac{1}{2})} \int_{0}^{+\infty}\frac{1}{(4\tau)^\frac{1+a}{2} \Gamma \left(\frac{1+a}{2}\right)}y^{\frac{1+a}{2}-1} \mathrm{e}^{-\frac{1}{2\tau} y} \d y \, \\ &=4^{\frac{a}{2}}\tau^{\frac{a}{2}}\Gamma(\tfrac{1+a}{2})\Gamma(\tfrac{1}{2})^{-1}.\qedhere
\end{align*} 
\end{proof}

\begin{rem}\label{rem:ath_mom_bounds}\rm 
If $a=2$, from the proof of Lemma~\ref{lemma:ath_mom_bounds}, we infer that for $Z \sim g_{\tau}$ it holds $\Ex |Z|^2 =2\tau d$. 
\end{rem}

\begin{lem}\label{lem:prox-contracts-1}
If $V$ is a differentiable and $\lambda_V$-convex function with $\lambda_V\geq 0$, then, for any $\tau>0$ and $x,z\in\R^d$, it holds
$$
2\tau\big(V(\proxv(x))-V(z)\big)\leq |x-z|^2-|x-\proxv(x)|^2-(1+\tau\lambda_V)|\proxv(x)-z|\,.
$$
\end{lem}

\begin{proof}
We denote $y=\proxv (x)$. Due to the $\lambda_V$-convexity of $V$ we know that 
\[
V(y)-V(z)\leq \nabla V(y) \cdot (y-z)-\tfrac{\lambda_V}{2}|y-z|^2\,.
\]
By the definition of $\proxv$ it follows that $\nabla V(y)=\tfrac{1}{\tau}(x-y)$. In turn
\[
2\tau\big(V(y)-V(z)\big)\leq 2(x-y)\cdot (y-z)-\tau\lambda_V|y-z|^2\,.
\]
To complete the result it suffices to note that $
|x-z|^2=|x-y|^2+|y-z|^2+2(x-y)\cdot(y-z)
$.
\end{proof}

\begin{lem}\label{lem:contr_wasserstein}
If $V$ is a differentiable and $\lambda_V$-convex function with $\lambda_V\geq 0$, then for any \( \mu, \nu \in \mathcal{P}_2 \), we have  
\[
W_2\big((\proxv)_\#\mu, (\proxv)_\#\nu\big) \leq \tfrac{1}{1+\tau\lambda} W_2(\mu,\nu)\,.
\]  
\end{lem}
\begin{proof}  
Let \( X_1 \sim \mu \) and \( X_2 \sim \nu \) be an optimal coupling, and define \( Y_1 = \proxv(X_1) \) and \( Y_2 = \proxv(X_2) \). By the \( \lambda \)-convexity of \( V \), we have  
\[
(\nabla V(Y_1) - \nabla V(Y_2)) \cdot (Y_1 - Y_2) \geq \lambda |Y_1 - Y_2|^2\,.
\]  
Since \( Y_1 = X_1 - \tau \nabla V(Y_1) \) and \( Y_2 = X_2 - \tau \nabla V(Y_2) \), it follows that  
\[
(X_1 - Y_1) - (X_2 - Y_2) = \tau (\nabla V(Y_2) - \nabla V(Y_1))\,.
\]
We see that
\[
((X_1 - Y_1) - (X_2 - Y_2))\cdot (Y_1 - Y_2)  = \tau (\nabla V(Y_2) - \nabla V(Y_1))\cdot (Y_1 - Y_2)\geq \tau\lambda |Y_1 - Y_2|^2 \,.
\] and by the Cauchy–Schwarz inequality it holds
\[((X_1 - Y_1) - (X_2 - Y_2))\cdot (Y_1 - Y_2)=(X_1-X_2)\cdot(Y_1-Y_2)-|Y_1-Y_2|^2 \leq |X_1 -X_2|\,|Y_1 - Y_2|-|Y_1-Y_2|^2 \]
Altogether, by reordering terms, we obtain  
\[
\tfrac{1}{1+\tau\lambda} |X_1 - X_2| \geq |Y_1 - Y_2|\,.
\]  
Taking expectations on both sides gives  
\[
\left( \tfrac{1}{1+\tau\lambda} \right)^2 W_2^2(\mu,\nu) \geq \Ex |Y_1 - Y_2|^2\geq W_2^2((\proxv)_\#\mu, (\proxv)_\#\nu) \,,
\]  
which concludes the proof.  
\end{proof}   

\begin{coro}
\label{cor:prox-contracts}
  If $V$ is a differentiable and $\lambda_V$-convex function with $\lambda_V\geq 0$ that attains its minimum at $x^*=0$, then, for any $\tau>0$ and $x\in\R^d$,
    \[
    |\proxv(x)|\leq \tfrac{1}{1+\tau\lambda_V}|x|\,.
    \]
\end{coro}
\begin{rem}\rm
If the minimum of $V$ was not at $x^*=0$, then the result would have been
    \[
    |\proxv(x)-x^*|\leq \tfrac{1}{1+\tau\lambda_V}|x-x^*|\,.
    \]
\end{rem}
\begin{proof}
Taking $\mu := \delta_x$ and $\nu := \delta_{x^*} $ in Lemma~\ref{lem:contr_wasserstein} and using the fact that $W_2 (\delta_x, \delta_{x^*}) = |x-x^*| $.
\end{proof}

\begin{lem}\label{lem:wass-pert}
Suppose $X\sim \mu$, $Y=X+\xi$ for some $\xi\in\R^d$ and $Y\sim \nu$. Then
\begin{equation*}
W_2^2 (\mu,\nu) \leq |\xi|^2 \, .
\end{equation*}
\end{lem}

\begin{proof}It suffices to note that $
W_2^2 (\mu, \nu) =\inf_{X \sim \mu , Y \sim \nu} \Ex \lvert X - Y \rvert^2 \leq \Ex \lvert \xi \rvert^2 = |\xi|^2 \, .$
\end{proof}

\begin{lem}\label{lem:minimizer}
The $\lambda$-convex and lower semicontinuous functional $\cF$ given by \eqref{cF} admits a unique minimizer.
\end{lem}

\begin{proof}
We follow the second part of Lemma 2.4.8 in \cite{Ambrosio-Gigli-Savare}. Take a sequence $\nu_n\in \cP_2(\R^d)$ such that $\cF[\nu_n]\leq \inf_{\cP_2(\R^d)}\cF[\cdot]+\omega_n$ with $\omega_n\to0$ as $n\to+\infty$. Then by the $\lambda$-convexity along generalized geodesics, we get 
\begin{align*}
\tfrac{\lambda}{8}W_2^2(\nu_n,\nu_m)&\leq \tfrac{1}{2}\cF[\nu_n]+\tfrac{1}{2}\cF[\nu_m]-\cF[\nu_{\frac{1}{2}}^{n\rightarrow m}]\leq \tfrac{1}{2}(\omega_n+\omega_m)+\inf_{\cP_2(\R^d)}\cF[\cdot]-\cF[\nu_{\frac{1}{2}}^{n\rightarrow m}]\leq \tfrac{1}{2}(\omega_n+\omega_m).
\end{align*}
The last inequality follows by the property of infimum. Hence, $\nu_n$ is Cauchy in $W_2$, and then there exist a $\nu\in \cP_2(\R^d)$ such that $\nu=\inf_{\cP_2(\R^d)}\cF[\cdot]$. 
\end{proof}

\begin{lem}\label{lem:power-est}
Let $X,Z\in\cP_a(\rn),$ $a\geq 2$, be arbitrary and $Z$ be independent of $X$ and such that $\Ex[Z]=0$. Then
\begin{equation}\label{eq:taylor_exp}
\Ex | X + Z |^a \leq \Ex | X |^a + C_a (\Ex |X|^{a-2} |Z|^2 + \Ex |Z|^a)\,.    
\end{equation}
\end{lem}
\begin{proof}
We compute the Taylor expansion of first order of the function $f(z) = |\xi + z|^a, a \geq 2$ around zero with reminder estimation. To do that, we compute the first and second derivative 
\begin{flalign*}
\nabla f (z) &= a |\xi + z|^{a-1} \tfrac{(\xi + z)}{|\xi + z|} =   a |\xi + z|^{a-2} (\xi + z) \, , \\
\nabla^2 f (z) &= 2 (\tfrac{a}{2} -1)a |\xi+z|^{a-4} (\xi+z)\otimes (\xi+z) + \id \, a |\xi+z|^{a-2} \,.
\end{flalign*}
We make use of the Taylor expansion with the Cauchy reminder such that 
\begin{equation*} 
f(z) \leq f(0) + \nabla f(0)\cdot z + |\tfrac{1}{2} z^\top \nabla^2 f(\theta z) z |  \quad\text{for some }\ \theta \in [0,1] \, .
\end{equation*}
We take $z=Z$, $\xi=X$, and consider $\Ex [f(Z)]$. Then~\eqref{eq:taylor_exp} follows from elementary manipulation.
\end{proof}

The following lemma might be inferred from \cite[Lemma~8]{carrillo_prop_chaos_lambda} under more strict smoothness assumptions.

\begin{lem}\label{lem:lambda-saving}
Let  $V$ be $\lambda_V$-convex and $W$ be $\lambda_W$-convex such that $\lambda_V,  \lambda_W \in \mathbb{R}$. Then the function $\boldsymbol{\Psi}$ is $\min \{\lambda_V, \lambda_V + 2 \lambda_W\}$-convex.
\end{lem}

\begin{proof} 
We note that 
\begin{flalign*}
\boldsymbol{\Psi} (\theta \boldsymbol{x} + &(1-\theta)\boldsymbol{y}) = \sum_{i=1}^N\left(V(\theta x^i+(1-\theta)y^i) +  \frac{1}{2N}\sum_{\substack{j=1}}^{N}  W (\theta(x^i-x^j)+(1-\theta)(y^i-y^j))\right) \\ 
&\leq \sum_{i=1}^N \bigg( \theta \, V(x^i) +(1-\theta) V(y^i)  - \frac{\lambda_V}{2} \theta (1-\theta) \big\lvert x^i - y^i \big\rvert^2 \\ 
&\qquad+ \frac{1}{2N}\sum_{\substack{j=1}}^{N}  \theta \, W (x^i-x^j) + (1-\theta)\, W(y^i-y^j) - \frac{\lambda_W}{2} \theta(1-\theta) \big\lvert (x^i - x^j) - (y^i - y^j) \big\rvert^2 \bigg)\,.
\end{flalign*}
We consider the metric on the product space such that iff $\Rd$ is a~space with Euclidean metric, then the product space $\mathbb{R}^{Nd}$ has the metric given by the relation 
\begin{equation*}
\big|\boldsymbol{x}-\boldsymbol{y} \big|^2 = \sum_{i=1}^N \lvert x^i-y^i \rvert^2 ; \quad x^i,\, y^i \in \Rd, \quad \boldsymbol{x} = (x^1, \ldots, x^N), \, \boldsymbol{y} = (y^1, \ldots, y^N)  \in \mathbb{R}^{Nd} \, .
\end{equation*}
Then we apply this and consider that the norm on the right-hand side is zero for $i=j$
\begin{flalign*}
\boldsymbol{\Psi} (\theta \boldsymbol{x} + (1-\theta)\boldsymbol{y}) &\leq \theta \boldsymbol{\Psi}(\boldsymbol{x}) + (1-\theta) \boldsymbol{\Psi} (\boldsymbol{y}) - \frac{\lambda_V}{2} \theta(1-\theta) \, \big\lvert \boldsymbol{x} - \boldsymbol{y} \big\rvert^2  \\ &\qquad - \frac{1}{N} \sum_{\substack{i,j=1 \\ i \neq j}}^{N} \frac{\lambda_W}{2} \theta(1-\theta)  \big\lvert (x^i - y^i) - (x^j - y^j) \big\rvert^2 \, .
\end{flalign*}
Let us split it into the cases $\lambda_W \geq 0$ and $\lambda_W < 0$.
\begin{itemize}
\item [$\lambda_W \geq 0$:] We consider that term $\tfrac{\lambda_W}{2} \theta (1-\theta) \lvert (x^i-y^i) - (x^j-y^j) \rvert^2$ is nonnegative and we can conclude that the function $\boldsymbol{\Psi}$ is $\lambda_V$-convex.
\item [$\lambda_W < 0$:] We consider the triangle inequality $|a-b| \leq 2 |a|^2 + 2 |b|^2 $ and recall that for this case $-\lambda_W = |\lambda_W|$ such that 
\begin{align*}
&- \frac{1}{N} \sum_{\substack{i,j=1 \\ i \neq j}}^{N} \frac{\lambda_W}{2} \theta (1-\theta) \big\lvert (x^i-y^i) - (x^j-y^j) \big\rvert^2\leq - \frac{1}{N} \sum_{\substack{i,j=1 \\ i \neq j}}^{N} \frac{\lambda_W}{2} \theta (1-\theta) \big(2\lvert x^i-y^i \rvert^2 + 2\lvert x^j-y^j \rvert^2 \big) \\ 
&\qquad\quad\leq 
-  \frac{2\lambda_W}{2} \theta (1-\theta) \frac{1}{N} \sum_{\substack{i,j=1 \\ i \neq j}}^{N} \big(\lvert x^i-y^i \rvert^2 + \lvert x^j-y^j \rvert^2 \big)
\leq -\frac{2 \lambda_W}{2} \theta(1-\theta) \lvert x-y \rvert^2 \, .\qedhere 
\end{align*}
\end{itemize}
\end{proof}


\section*{Acknowledgments} 
The authors were supported by NCN (Narodowe Centrum Nauki) grant 2018/31/B/ST1/00253.

During the final preparations of this paper, Benko was hosted by the Department of Mathematical Sciences at NTNU. The travel was partially supported by the project Pure Mathematics in Norway, funded by Trond Mohn Foundation and Troms{\o} Research Foundation.
Benko was also supported by internal grant of specific research FSI-S-23-8161. The authors acknowledge also the support of Initiative of Excellence at University of Warsaw.

Finally, we would like to thank the referee for suggestions to improve the paper. 

\bibliographystyle{abbrv}
\bibliography{bib}

\end{document}